\documentclass[]{amsart}
\usepackage{amsmath}
\usepackage{bm}
\usepackage{amsfonts}
\usepackage{amssymb}
\usepackage{txfonts}
\usepackage{mathtools}
\usepackage[T1]{fontenc}
\newtheorem{theorem}{Theorem}
\newtheorem{corollary}{Corollary}
\newtheorem{definition}{Definition}
\newtheorem{example}{Example}
\newtheorem{lemma}{Lemma}

\newtheorem{remark}{Remark}
\newcommand{\eps}{\varepsilon}

\DeclareMathOperator{\dist}{dist}
\DeclareMathOperator{\fix}{Fix}
\DeclareMathOperator{\gr}{Gr}
\DeclareMathOperator{\co}{co}

\DeclareMathOperator{\Int}{int}
\DeclareMathOperator{\cl}{cl}

\DeclareMathOperator{\esssupp}{ess\;supp}

\newcommand{\w}{\tilde}
\newcommand{\map}{\multimap}
\newcommand{\loc}{L^1_{\rm loc}(\Omega,E)}
\newcommand{\eL}[2]{L^{#1}_{{\rm loc}}\left(#2,E\right)}
\newcommand{\<}{\leqslant}
\newcommand{\f}{\left}

\newcommand{\g}{\right}
\newcommand{\n}{{n\geqslant 1}}
\newcommand{\K}{{k\geqslant 1}}
\newcommand{\z}[1]{(#1)_{n=1}^\infty}
\newcommand{\x}[1]{\{#1\}_{n=1}^\infty}
\DeclareMathOperator{\F}{F}
\DeclareMathOperator{\ka}{k}
\DeclareMathOperator{\gie}{g}
\newcommand{\R}[1]{\mathbb{R}^{#1}}

\newcommand{\res}[2]{#1\:\rule[-1.5mm]{0.45pt}{4mm}\,\rule[-1mm]{0mm}{4mm}_{#2}}

\begin{document}
\title[Existence of solutions for a class of multivalued functional integral equations]{Existence of solutions for a class of multivalued functional integral equations of Volterra type via the measure of nonequicontinuity on the Fr\'echet space ${\bf C(\Omega,E)}$}
\author{Rados\l aw Pietkun}
\address{Toru\'n, Poland}
\email{rpietkun@pm.me}
\begin{abstract}
The existence of continuous not necessarily bounded solutions of nonlinear functional Volterra integral inclusions in infinite dimensional setting is shown with the aid of the measure of nonequicontinuity. New abstract topological fixed point results for admissible condensing operators are introduced. Weak compactness criterion in the space of locally integrable functions in the sense of Bochner is set forth. Some examples illustrating the usefulness of the presented approach are also included.
\end{abstract}
\maketitle
 
\section{Introduction}
There is a long practice of proving the existence of continuous solutions to integral equations of Volterra type. The authors of \cite{banas} came up with the idea of application of a measure of non-compactness defined on $BC(\R{}_+)$ to demonstrate the existence of solutions to Volterra integral equation of the form 
\begin{equation}\label{banas}
x(t)=f(t,x(t))+\int_0^tu(t,s,x(s))\,ds,\;\;\;t\geqslant 0.
\end{equation}
This approach turned out to be very prolific and resulted in many articles patterned on the above, to a greater or lesser extent. The papers \cite{alla,agh3,agh,agh2,arab,samadi,khandani} are focused on the case of a scalar univalent equation and they narrow the solutions' search region to the Banach space $BC(\Omega)$. In \cite{das1,das2,rizvi} the fixed point approach was used to obtain solutions of functional integral equations in sequence spaces $c_0$ and $\ell_1$. In this article, we get rid of the assumption of one-dimensionality and univalency of the Volterra equation and we allow the existence of unbounded solutions. Considered here set-valued variant of equation \eqref{banas} has basically the following form
\begin{equation}\label{general}
u(x)\in G\f(x,u(x),\int_{\Lambda(x)}k(x,y)F(y,u(y))\,dy\g),\;\;x\in\Omega
\end{equation}
with $G\colon\Omega\times E\times E\map E$, $F\colon\Omega\times E\map E$ and $\Lambda\colon\Omega\subset\R{N}\to{\mathfrak L}(\R{N})$. Caused by technical and competency restrictions we formulate sufficient conditions for the existence of continuous solutions to inclusion \eqref{general} in three particular cases framed in equations \eqref{inclusion2}, \eqref{volterra} and \eqref{hj}. The proofs of theorems regarding these equations boil down to the showing of fixed point existence of suitable operators, whose admissibility allows the application of Sadovski\u{\i} type fixed point result (Theorem \ref{fixed}). The assumption $(\F_5)$ regarding the multivalued perturbation $F$ poses a substitute of compactness in the space $E$ which, along with the quasi-Lipschitzeanity of the external operator $g$ gives the opportunity of showing that the superposition $N_g\circ(I\times(V\circ N_F))$ of the Nemytski\v{\i} operators $N_g$ and $N_F$ with the integral Volterra operator $V$ is condensing with respect to some measure of non-compactness defined on the Fr\'echet space $C(\Omega,E)$. In the existing situation, it is quite natural to accept that the Nemytski\v{\i} operator $N_F$ maps the space $C(\Omega,E)$ onto the Fr\'echet space $\loc$. The justification of the admissibility of operator $V\circ N_F$ forces the formulation of legible criteria of weak compactness in the space $\loc$. This was done in Theorem \ref{extension}. Taking into account some specific assumptions regarding the geometry of the Banach space $E$, this result generalizes the well-known Theorem \ref{ulger}. The article is complemented by four examples well illustrating the advantage of the formulated results over those published previously.\vspace{\baselineskip}
\par Let us introduce some notations which will be used in this paper. Let $(E,|\cdot|)$ be a Banach space, $E^*$ its normed dual and $(E,w)$ the space $E$ furnished with the weak topology. \par The normed space of bounded linear operators $S\colon E\to E$ is denoted by $\mathcal{L}(E)$. Given $S\in\mathcal{L}(E)$, $||S||_{{\mathcal L}}$ is the norm of $S$. For any $\eps>0$ and $A\subset E$, $B(A,\eps)$ ($D(A,\eps)$) stands for an open (closed) $\eps$-neighbourhood of the set $A$. If $x\in E$ we put $\dist(x,A):=\inf\{|x-y|\colon y\in A\}$. Besides, for two nonempty closed bounded subsets $A$, $B$ of $E$ the symbol $h(A,B)$ stands for the Hausdorff distance from $A$ to $B$, i.e. $h(A,B):=\max\{\sup\{\dist(x,B)\colon x\in A\},\sup\{\dist(y,A)\colon y\in B\}\}$.\par We use symbols of functional spaces, such as $C(\Omega,E)$, $\loc$, $L^\infty(\Omega,E^*)$, $H^2(\R{n})$, $(L^p(\Omega,E),||\cdot||_p)$, in their commonly accepted meaning.\par Given metric space X, a set-valued map $F\colon X\map E$ assigns to any $x \in X$ a nonempty subset $F(x)\subset E$. $F$ is (weakly) upper semicontinuous, if the small inverse image $F^{-1}(A)=\{x\in X\colon F(x)\subset A\}$ is open in $X$ whenever $A$ is (weakly) open in $E$. We say that $F\colon X\map E$ is upper hemicontinuous if for each $x^*\in E^*$, the function $\sigma(x^*,F(\cdot))\colon X\to\R{}\cup\{+\infty\}$ is upper semicontinuous (as an extended real function), where $\sigma(x^*,F(x))=\sup\limits_{y\in F(x)}\langle x^*,y\rangle$. We have the following characterization: a map $F\colon X\map E$ with convex values is weakly upper semicontinues and has weakly compact values iff given a sequence $(x_n,y_n)$ in the graph $\gr(F)$ of map $F$ with $x_n\xrightarrow[n\to\infty]{X}x$, there is a subsequence $y_{k_n}\xrightharpoonup[n\to\infty]{E}y\in F(x)$ ($\rightharpoonup$ denotes the weak convergence). The set of all fixed points of the map $F\colon E\map E$ is denoted by $\fix(F)$.\par Let $H^\ast(\cdot)$ denote the Alexander-Spanier cohomology functor with coefficients in the field of rational numbers ${\mathbb Q}$ (see \cite{spanier}). We say that a topological space $X$ is acyclic if the reduced cohomology $\w{H}^q(X)$ is $0$ for any $q\geqslant 0$. %A compact (nonempty) metric space $X$ is an $R_\delta$-set if it is the intersection of a decreasing sequence of compact contractible metric spaces. In particular, $R_\delta$-sets are acyclic.
\par An upper semicontinuous map $F\colon E\map E$ is called acyclic if it has compact acyclic values. A set-valued map $F\colon E\map E$ is admissible (in the sense of \cite[Def.40.1]{gorn}) if there is a Hausdorff topological space $\Gamma$ and two continuous functions $p\colon\Gamma\to E$, $q\colon\Gamma\to E$ from which $p$ is a Vietoris map such that $F(x)=q(p^{-1}(x))$ for every $x\in E$. Clearly, every acyclic map is admissible. Moreover, the composition of admissible maps is admissible (\cite[Th.40.6]{gorn}).
\par A real function $\beta$ defined on the family of bounded subsets $\Omega$ of $E$ defined by the formulae \[\beta(\Omega):=\inf\{\eps>0:\Omega\mbox{ admits a finite covering by balls of a radius }\eps\}\]
is called the Hausdorff measure of non-compactness (MNC). Recall that this measure is regular, monotone, nonsingular, semi-additive, algebraically semi-additive and invariant under translation (for details see \cite{sadovski}).

\section{Fixed point results}
Our fixed point results rely on the concept of an abstract measure of non-compactness. That is why we will start from
\begin{definition}\label{mnc}
A set function $\mu\colon{\mathcal B}(\mathbb{F})\to P$, defined on the family ${\mathcal B}(\mathbb{F})$ of bounded subsets of the Fr\'echet space $\mathbb{F}$ with values in a positive cone $P$ of some partially ordered vector space $(E,\geqslant)$, is called a measure of non-compactness, if the following conditions are satisfied:
\begin{itemize}
\item[(i)] $\mu(\{x_0\}\cup\Omega)=\mu(\Omega)$ for every $x_0\in\mathbb{F}$ and every $\Omega\in{\mathcal B}(\mathbb{F})$,
\item[(ii)] $\mu(\overline{\Omega})=\mu(\Omega)$ for every $\Omega\in{\mathcal B}(\mathbb{F})$,
\item[(iii)] $\mu(\co\,\Omega)=\mu(\Omega)$ for every $\Omega\in{\mathcal B}(\mathbb{F})$.
\end{itemize}
\end{definition}
Having established axioms of the measure $\mu$, we can formulate fixed point theorems for admissible condensing set-valued operators defined on the Fr\'echet space:
\begin{theorem}\label{fixed}
Let $X$ be a nonempty closed convex and bounded subset of a Fr\'echet space $\mathbb{F}$ and $\mu\colon{\mathcal B}(\mathbb{F})\to P$ an MNC on $\mathbb{F}$ in the sense of Definition \ref{mnc}. Assume that $F\colon X\map X$ is an admissible set-valued operator satisfying
\begin{equation}\label{cond}
\exists\,f\in\Phi\,\forall\,\Omega\subset X\;\;\overline{\Omega}\,\text{ noncompact}\Rightarrow f\f(\,\mu(F(\Omega)),\mu(\Omega)\g)\in P\setminus\{0\},
\end{equation}
where \[\Phi:=\f\{f\colon P^2\to(E,\geqslant)\colon\Delta_{P^2}\subset f^{-1}(-P)\g\}.\]
Then $\fix(F)$ is nonempty and compact.
\end{theorem}

\begin{proof}
Fix an arbitrary $x_0\in X$. Consider a family $\{T_\alpha\}_{\alpha\in A}$ of all fundamental subsets of the multimap $F$ containing $x_0$. Recall after Krasnosel'ski\u{\i} that the closed convex set $T\subset X$ is fundamental if $F(T)\subset T$ and for any $x\in X$, it follows from $x\in\overline{\co}\f(F(x)\cup T\g)$ that $x\in T$. Observe that family $\{T_\alpha\}_{\alpha\in A}$ is nonempty (take for example $X$). Define $T:=\bigcap\limits_{\alpha\in A}T_\alpha$. Next, note that $T$ and $\overline{\co}(F(T)\cup\{x_0\})$ are fundamental. Whence, $T=\overline{\co}(F(T)\cup\{x_0\})$.\par If $T$ is noncompact, then $f\f(\,\mu(F(T)),\mu(T)\g)\in P\setminus\{0\}$ for some $f\in\Phi$. Invoking the very definition of an MNC (Definition \ref{mnc}.), we arrive at \[\mu(F(T))=\mu(\{x_0\}\cup F(T))=\mu(\overline{\co}(\{x_0\}\cup F(T)))=\mu(T).\] The latter means that $f\f(\,\mu(F(T)),\mu(T)\g)\in-P$, in view of the definition of the class $\Phi$. We reached the contradiction, since $P$ is pointed. Consequently, $T$ must be compact.\par By virtue of the Dugundji Extension Theorem the domain $T$ is an absolute extensor for the class of metrizable spaces. Therefore, the set-valued map $F\colon T\map T$ must have at least one fixed point $x\in T$, in view of \cite[Th.7.4]{gorn2}. Moreover, $\fix(F)$ forms a closed subset of the compact domain $T$.
\end{proof}

The corresponding continuation variant of the above fixed point theorem  contains the following:

\begin{theorem}\label{fixed2}
Let $X$ be a nonempty closed convex and bounded subset of a Fr\'echet space $\mathbb{F}$ and $\mu\colon{\mathcal B}(\mathbb{F})\to P$ an MNC on $\mathbb{F}$ within the meaning of Definition \ref{mnc}, which has an additional property of being monotone. Assume that $U$ is relatively open in $X$ and its closure is a retract of $X$. Assume further that $F\colon\overline{U}\map X$ is an admissible set-valued map and for some $x_0\in U$ the following two conditions are satisfied:
\begin{equation}\label{cond2}
\exists\,f\in\w{\Phi}\,\forall\,\Omega\subset U\;\;\overline{\Omega}\,\text{ noncompact}\Rightarrow f\f(\,\mu(F(\Omega)),\mu(\Omega)\g)\in P\setminus\{0\},
\end{equation}
where \[\w{\Phi}:=\f\{f\colon P^2\to(E,\geqslant)\colon\{(x,y)\in P^2\colon x-y\in P\}\subset f^{-1}(-P)\g\},\] and
\begin{equation}\label{LS}
x\not\in(1-\lambdaup)x_0+\lambdaup F(x)\text{ on }\overline{U}\setminus U\text{ for all }\lambdaup\in(0,1).
\end{equation}
Then $\fix(F)$ is nonempty and compact.
\end{theorem}

\begin{proof}
Keeping the notation and notions contained in the proof of \cite[Th.3.]{pietkun2}, consider a family $\{T_\alpha\}_{\alpha\in A}$ of all fundamental subsets of previously defined multimap $\w{F}\colon X\map X$ containing $x_0$. Let $T:=\bigcap\limits_{\alpha\in A}T_\alpha$. As we have noted previously, $T=\overline{\co}(\w{F}(T)\cup\{x_0\})$.\par If $\overline{T\cap U}$ is noncompact, then $f(\mu(F(T\cap U)),\mu(T\cap U))\in P\setminus\{0\}$, by \eqref{cond2}. On the other hand \[T\cap U\subset T=\overline{\co}(\w{F}(T)\cup\{x_0\})=\overline{\co}(F(T\cap U)\cup\{x_0\}),\] which means that $\mu(F(T\cap U))-\mu(T\cap U)\in P$. Thus, $f(\mu(F(T\cap U)),\mu(T\cap U))\in-P$, by the very definition of the class $\w{\Phi}$. Therefore, $\overline{T\cap U}$ must be compact. Since $F$ is admissible, $T$ is compact as well. As we have seen in the proof of \cite[Th.3.]{pietkun2}, $\w{F}\colon T\map T$ is also an admissible multimap.\par Once more, in view of \cite[Th.7.4]{gorn2}, the set-valued map $\w{F}\colon T\map T$ must have at least one fixed point $x\in T$. Observe that $\fix(\w{F})=\fix(F)$.
\end{proof}

\begin{example}\label{ex}
Let $E:=\R{\mathbb{N}}$ be a linear space of all scalar valued sequences endowed with the natural pointwise order and $P=\R{\mathbb{N}}_+$. Define $f\colon P^2\to(E,\geqslant)$ in the following way
\begin{equation}\label{f}
f\f(\z{x_n},\z{y_n}\g):=\z{k_ny_n-x_n},
\end{equation}
where $\z{k_n}\in(0,1)^{\mathbb{N}}$. Then $f\in\w{\Phi}\subset\Phi$.
\end{example}

\section{Weak compactness in $\loc$}
The most known up to date result regarding weak compactness in the Bochner space $L^1(E)$ is the following conclusion stemming from the celebrated Rosenthal's dichotomy theorem:
\begin{theorem}[\protect{\cite[Cor.9]{ulger}}]\label{ulger}
Let $(\Omega,\Sigma,\mu)$ be a finite measure space with $\mu$ being a nonatomic measure on $\Sigma$. Let $A$ be a uniformly $p$-integrable subset of $L^p(\Omega,E)$ with $p\in[1,\infty)$. Assume that for a.a. $\omega\in\Omega$, the set $\{f(\omega)\colon f\in A\}$ is relatively weakly compact in $E$. Then $A$ is relatively weakly compact.
\end{theorem}
With the aid of the Grothendieck's lemma and the following generalization of the Riesz representation theorem 
\begin{theorem}[\protect{\cite[Th.3.2.]{uhl}}]\label{uhl}
Let $p\in[1,\infty)$ and $p^{-1}+q^{-1}=1$. If $(\Omega,\Sigma,\mu)$ is a $\sigma$-finite measure space and $E$ is a Banach space such that $E^*$ has the Radon-Nikodym property, then $L^p(\Omega,E)$ and $L^q(\Omega,E^*)$ are isometrically isomorphic under the correspondence $l\in L^p(\Omega,E)^*\leftrightarrow g\in L^q(\Omega,E^*)$ defined by \[l(f)=\int_\Omega\langle f(\omega),g(\omega)\rangle\,\mu(d\omega),\;\;\;f\in L^p(\Omega,E).\]
\end{theorem}
\noindent we are able to prove Theorem \ref{ulger} in the context of a $\sigma$-finite measure space.
\begin{theorem}\label{rwc}
Let $p\in[1,\infty)$. Let $(\Omega,\Sigma,\mu)$ be a $\sigma$-finite measure space with $\mu$ being a nonatomic measure on $\Sigma$. Assume that $E$ is a Banach space such that $E^*$ has the Radon-Nikodym property. If $A$ is a uniformly $p$-integrable subset of $L^p(\Omega,E)$ with relatively weakly compact cross-sections $A(\omega)$ for a.a. $\omega\in\Omega$, then $A$ is relatively weakly compact.
\end{theorem}

\begin{proof}
Let $\{\Omega_k\}_{k=1}^\infty$ be an increasing family such that $\Omega=\bigcup\limits_{k=1}^\infty\Omega_k$ and $\mu(\Omega_k)<\infty$. Clearly, $\f(\Omega_k,\Omega_k\cap\Sigma,\res{\mu}{\Omega_k\cap\Sigma}\g)$ is a finite measure space with $\res{\mu}{\Omega_k\cap\Sigma}$ being nonatomic measure on $\Omega_k\cap\Sigma$. In view of {\cite[Cor.9]{ulger}}, $A_k:=\res{A}{\Omega_k}$ is a relatively weakly compact subset of $L^p\f(\Omega_k,\Omega_k\cap\Sigma,\res{\mu}{\Omega_k\cap\Sigma};E\g)$. Put $\w{A}_k:=\{f{\bf 1}_{\Omega_k}\colon f\in A_k\}$. Consider an arbitrary $\z{f_n{\bf 1}_{\Omega_k}}\subset\w{A}_k$. We may assume, passing to a subsequence if necessary, that $f_n\xrightharpoonup[n\to\infty]{L^p(\Omega_k,E)}f\in A_k$. Take $\zeta\in L^p(\Omega,E)^*$. In view of \cite[Th.3.2.]{uhl}, there exists $\xi\in L^{\frac{p}{p-1}}(\Omega,E^*)$ such that 
\begin{align*}
\langle \zeta,f_n{\bf 1}_{\Omega_k}\rangle&=\int\limits_\Omega\langle\xi(\omega),f_n{\bf 1}_{\Omega_k}(\omega)\rangle\,d\mu\\&=\int\limits_{\Omega_k}\langle\res{\xi}{\Omega_k}(\omega),f_n(\omega)\rangle\,d\mu\xrightarrow[n\to\infty]{}\int\limits_{\Omega_k}\langle\res{\xi}{\Omega_k}(\omega),f(\omega)\rangle\,d\mu=\int\limits_{\Omega}\langle\xi(\omega),\res{f}{\Omega_k}(\omega)\rangle\,d\mu
\end{align*}
with $\res{\xi}{\Omega_k}\in L^{\frac{p}{p-1}}(\Omega_k,E^*)$. Hence, $f_n{\bf 1}_{\Omega_k}\xrightharpoonup[n\to\infty]{L^p(\Omega,E)}f{\bf 1}_{\Omega_k}\in\w{A}_k$. In other words, the set $\w{A}_k$ is relatively weakly compact in $L^p(\Omega,E)$.\par Since $\mu(\Omega\setminus\Omega_k)\xrightarrow[k\to\infty]{}0$, we have \[\lim_{k\to\infty}\sup_{f\in A}\int\limits_{\Omega\setminus\Omega_k}|f(\omega)|^p\,d\mu=0.\] The latter means that for all $\eps>0$ we can find $k_0\in\mathbb{N}$ such that $\int\limits_{\Omega\setminus\Omega_{k_0}}|f(\omega)|^p\,d\mu<\eps$ for all $f\in A$. Fix $\eps>0$. One easily sees that \[\sup_{f\in A}\f\Arrowvert f-\f(\res{f}{\Omega_{k_0}}\g){\bf 1}_{\Omega_{k_0}}\g\Arrowvert_p^p=\sup_{f\in A}\int\limits_{\Omega\setminus\Omega_{k_0}}|f(\omega)|^p\,d\mu<\eps.\] Consequently, $A\subset\w{A}_{k_0}+B(0,\eps)$ with $\w{A}_{k_0}$ being relatively weakly compact in $L^p(\Omega,E)$. By virtue of Grothendieck's Lemma, the set $A$ must be relatively weakly compact. 
\end{proof}

Let $\Omega\subset\R{N}$ be open (not necessarily bounded) and ${\mathfrak L}(\R{N})$ be the Lebesgue $\sigma$-field. Let $\rho\colon{\mathfrak L}(\R{N})\times{\mathfrak L}(\R{N})\to\R{}_+$ be a pseudometric, given by $\rho(A,B):=\ell(A\triangle B)$. Assume once and for all that $\Lambda\colon\Omega\to{\mathfrak L}(\R{N})$ is $\rho$-continuous and maps bounded sets into bounded subsets of $\Omega$. \par By the exhaustion of the domain $\Omega$ we mean any increasing sequence $\z{\Omega_n}$ of open bounded subsets, which cover $\Omega$. In this instance, the family of rings $\{\w{\Omega}_n\}_{n=1}^\infty$ with $\w{\Omega}_n:=\cl_{\Omega}(\Omega_n)\setminus\Omega_{n-1}$ poses a compact partitioning of the set $\Omega$. \par Our standing hypothesis on the space $E$ is the following:
\begin{itemize}
\item[$(\mathbb{E})$] $E$ and the bidual $E^{**}$ are strictly convex Banach spaces, while the dual $E^*$ has the Radon-Nikodym property.  
\end{itemize}

\begin{remark}
Reflexive Banach spaces meet assumption $(\mathbb{E})$ $($possibly after Troyanski's {re\-norming}$)$.
\end{remark}

\begin{lemma}\label{iso}
Let $\z{\Omega_n}$ be an exhaustion of $\Omega$. The spaces $\eL{p}{\Omega}$ and $\prod\limits_{n=1}^\infty L^p(\w{\Omega}_n,E)$ are isomorphic.
\end{lemma}

\begin{proof}
Define $\Phi\colon\eL{p}{\Omega}\to\prod\limits_{n=1}^\infty L^p(\w{\Omega}_n,E)$ by $\Phi(f)=\f(\res{f}{\w{\Omega}_n}\g)_{n=1}^\infty$. The only non-obvious property of the isomorphism $\Phi$ is surjectivity. Assume that $\z{f_n}\in\prod\limits_{n=1}^\infty L^p(\w{\Omega}_n,E)$. Let $f\colon\Omega\to E$ be given by $\res{f}{\w{\Omega}_n}:=f_n$ for $\n$. Since $f_n\in L^p(\w{\Omega}_n,E)$, there exists a sequence $(g_n^k)_{k=1}^\infty$ of simple functions such that $g_n^k(x)\xrightarrow[k\to\infty]{E}f_n(x)$ for $x\in\w{\Omega}_n\setminus I_n$ with $\ell(I_n)=0$. Let $g_k\colon\Omega\to E$ be given by $g_k:=\sum\limits_{n=1}^\infty g_n^k{\bf 1}_{\w{\Omega}_n}$. Obviously, $g_k$ is countably valued and strongly measurable. Since $g_k(x)\xrightarrow[k\to\infty]{E}f(x)$ for $x\in\Omega\setminus\bigcup\limits_{n=1}^\infty I_n$ with $\ell\f(\bigcup\limits_{n=1}^\infty I_n\g)=0$, the mapping $f$ must be strongly measurable. If $K\subset\Omega$ is compact, then there is $n\in\mathbb{N}$ such that $K\subset\cl_{\Omega}(\Omega_n)$ and \[\int\limits_K|f(x)|^p\,dx\<\int\limits_{\cl_\Omega(\Omega_n)}|f(x)|^p\,dx=\sum_{k=1}^n\int\limits_{\w{\Omega}_k}|f_k(x)|^p\,dx=\sum_{k=1}^n||f_k||_{L^p(\w{\Omega}_k,E)}^p<\infty.\] Whence, $f\in\eL{p}{\Omega}$ and $\Phi(f)=\z{f_n}$. 
\end{proof}

\begin{lemma}\label{conjugate}
Assume $(\mathbb{E})$. If \[L^\infty_{c}(\Omega,E^*):=\{g\in L^\infty(\Omega,E^*)\colon\esssupp(g)\text{ is compact }\},\] then
\[\loc^*=\f\{\varphi\colon\eL{1}{\Omega}\to\R{}\colon\exists\,g\in L^\infty_c(\Omega,E^*)\ni\varphi(f)=\!\int_\Omega\langle g(x),f(x)\rangle\,dx\g\}\!.\]
\end{lemma}

\begin{proof}
Assume that $\varphi\colon\eL{1}{\Omega}\to\R{}$ is given by $\varphi(f):=\int\limits_\Omega\langle g(x),f(x)\rangle\,dx$ for some $g\in L^\infty(\Omega,E^*)$ with $||\esssupp(g)||^+<+\infty$. Let $\z{\Omega_n}$ be any exhaustion of $\Omega$. Note that there must be an $n_0\in\mathbb{N}$ such that $\esssupp(g)\subset\cl_\Omega(\Omega_{n_0})$. Therefore,
\begin{align*}
\int\limits_\Omega\langle g(x),f(x)\rangle\,dx&=\int\limits_{\cl_\Omega(\Omega_{n_0})}\langle g(x),f(x)\rangle\,dx\<\int\limits_{\cl_\Omega(\Omega_{n_0})}|g(x)||f(x)|\,dx\\&\<||g||_{L^\infty(\Omega,E^*)}||f||_{L^1\big(\cl_\Omega(\Omega_{n_0}),E\big)}<\infty,
\end{align*}
which means that the value $\varphi(f)$ is well-defined. If $f_k\xrightarrow[k\to\infty]{\eL{1}{\Omega}}f$, then \[|\varphi(f_k)-\varphi(f)|\<||g||_{L^\infty(\Omega,E^*)}||f_k-f||_{L^1\big(\cl_\Omega(\Omega_{n_0}),E\big)}\xrightarrow[k\to\infty]{}0.\] Therefore, $\varphi\in\loc^*$. \par Now, let us assume that $\psi\in\eL{1}{\Omega}^*$. Since $L^1(\Omega,E)\hookrightarrow\eL{1}{\Omega}$ continuously, $\w{\psi}:=\res{\psi}{L^1(\Omega,E)}\in L^1(\Omega,E)^*$. In view of Theorem \ref{uhl}, there exists $g\in L^\infty(\Omega,E^*)$ such that \[\w{\psi}(f)=\int_\Omega\langle g(x),f(x)\rangle\,dx\] for $f\in L^1(\Omega,E)$.\par Let $J\colon E^*\map E$ be defined by $J(x^*):=\f\{x\in E\colon\langle x^*,x\rangle=|x|^2=|x^*|^2\g\}$, i.e. $J$ is the inverse of the duality map. Since $E$ is strictly convex, $J$ is a mapping. It can be shown that $J$ is demicontinuous. To this aim, assume that $x_n^*\xrightarrow[n\to\infty]{E^*}x^*_0$. The sequence $\z{J(x^*_n)}$ is relatively weak-$*$ compact in the 
double dual $E^{**}$, thanks to Banach-Alaoglu theorem. In other words, there exists $y\in E^{**}$ such that $\langle x^*,J(x^*_{k_n})\rangle\xrightarrow[n\to\infty]{}\langle y,x^*\rangle$ for every $x^*\in E^*$. On the one hand 
\begin{equation}\label{jeden}
\langle y,x^*_0\rangle=\lim\limits_{n\to\infty}\langle x^*_{k_n},J(x^*_{k_n})\rangle=\lim\limits_{n\to\infty}|x^*_{k_n}|^2=|x^*_0|^2.
\end{equation}
On the other 
\begin{equation}\label{dwa}
\langle y,x^*\rangle=\lim\limits_{n\to\infty}\langle x^*,J(x^*_{k_n})\rangle\<\lim\limits_{n\to\infty}|x^*||J(x^*_{k_n})|=|x^*||x^*_0|
\end{equation}
for every $x^*\in E^*$. From \eqref{jeden} and \eqref{dwa} it follows that $|x_0^*|\<|y|$ and $|y|\<|x_0^*|$, respectively. Thus, $\langle y,x_0^*\rangle=|x_0^*|^2=|y|^2$. This mean that $y\in{\mathcal F}(x_0^*)$ with ${\mathcal F}\colon E^*\map E^{**}$ being the duality map. Since $J(x_0^*)\in{\mathcal F}(x_0^*)$ and $E^{**}$ has strictly convex norm, one gets $y=J(x_0^*)$. Eventually, $J(x_n^*)\xrightharpoonup[n\to\infty]{E}J(x_0^*)$.\par Remind that $\w{\Omega}_n:=\cl_\Omega(\Omega_n)\setminus\Omega_{n-1}$. Suppose that there is a sequence $\f\{\w{\Omega}_{m_n}\g\}_{n=1}^\infty$ such that \[a_n:=\int_{\w{\Omega}_{m_n}}|g(x)|\,dx\neq 0.\] Since $J\circ g$ is weakly measurable and essentially separably valued, the strong measurability of $J\circ g\colon\Omega\to E$ follows by the Pettis measurability theorem. Observe that the formula 
\[f_n(x):=
\begin{cases}
a_k^{-1}|J(g(x))|^{-1}J(g(x))&\text{on }\w{\Omega}_{m_k}\text{ for }k=1,\ldots,n\\
0&\text{elsewhere}
\end{cases}\] 
makes sense almost everywhere on $\Omega$. Thus, $f_n\in L^1(\Omega,E)$. One easily sees that \[\w{\psi}(f_n)=\int_\Omega\langle f_n(x),g(x)\rangle\,dx=\sum_{k=1}^na_k^{-1}\int_{\w{\Omega}_{m_k}}|g(x)|\,dx=n.\] Whence $\psi(f_n)\xrightarrow[n\to\infty]{}+\infty$. Since $f_n\xrightarrow[n\to\infty]{\eL{1}{\Omega}}\sum\limits_{n=1}^\infty a_n^{-1}|J(g(\cdot))|^{-1}(J\circ g){\bf 1}_{\w{\Omega}_{m_n}}$, this contradicts the continuity of the functional $\psi$. Therefore, there is $N\in\mathbb{N}$ such that $\int_{\w{\Omega}_n}|g(x)|\,dx=0$ for all $n\geqslant N$. Hence, $\int_{\Omega\setminus\Omega_{N-1}}|g(x)|\,dx=0$, which means that $g(x)=0$ a.e. on $\Omega\setminus\cl_\Omega(\Omega_{N-1})$. It follows that $\esssupp(g)\subset\cl_\Omega(\Omega_{N-1})$, i.e. the support $\esssupp(g)$ must be bounded. Since the subspace $L^1(\Omega,E)$ is dense in $\eL{1}{\Omega}$, the functional $\psi$ constitutes an element of the set \[\f\{\varphi\colon\eL{1}{\Omega}\to\R{}\colon\exists\,g\in L^\infty_c(\Omega,E^*)\ni\varphi(f)=\!\int_\Omega\langle g(x),f(x)\rangle\,dx\g\}\!.\]
\end{proof}

The next result is a technical but crucial uplifting of Theorem \ref{rwc} onto the case of Bochner locally integrable functions.
\begin{theorem}\label{extension}
Assume $(\mathbb{E})$. Let $\Omega\subset\R{N}$ be open $($not necessarily bounded$)$ and ${\mathfrak L}(\R{N})$ be the Lebesgue $\sigma$-field. Let $A$ be a locally integrably bounded subset of $L^1_{\text{loc}}(\Omega,{\mathfrak L}(\R{N})\cap\Omega,\ell;E)$. Assume that for a.a. $x\in\Omega$, the set $\{f(x)\colon f\in A\}$ is relatively weakly compact in $E$. Then $A$ is relatively weakly compact.
\end{theorem}

\begin{proof}
Let $\z{\Omega_n}$ be any exhaustion of $\Omega$. Consider a net $(w_\sigma)_{\sigma\in\Sigma}\subset A$. Observe that $\f(\res{w_\sigma}{\w{\Omega}_n}\g)_{\sigma\in\Sigma}$ as a net in $L^1\f(\w{\Omega}_n,E\g)$ meets assumptions of Theorem \ref{ulger}. Thus, for each $n\in\mathbb{N}$ there exists a directed set $(\Sigma_n,\preccurlyeq_n)$ and a net $(w_{\sigma'})_{\sigma'\in\Sigma_n}$ finer than the net $(w_\sigma)_{\sigma\in\Sigma}$, which satisfies $\res{w_{\sigma'}}{\w{\Omega}_n}\xrightharpoonup[\sigma'\in\Sigma_n]{L^1(\w{\Omega}_n,E)}w^n$. We may assume w.l.o.g. that for every pair $(n,m)\in\mathbb{N}^2$ with $n\geqslant m$ the net $(w_{\sigma''})_{\sigma''\in\Sigma_n}$ is finer than $(w_{\sigma'})_{\sigma'\in\Sigma_m}$ i.e., there exists a nondecreasing function $\varphi_{nm}\colon\Sigma_n\to\Sigma_m$ such that
\begin{itemize}
\item[(i)] the range $\varphi_{nm}(\Sigma_n)$ is cofinal in $\Sigma_m$,
\item[(ii)] $\forall\,\sigma''\in\Sigma_n\;\;w_{\sigma''}=w_{\varphi_{nm}(\sigma'')}$.
\end{itemize}
In other words, we are dealing with an inverse system $\f\{(\Sigma_n,\preccurlyeq_n),\varphi_{nm}\colon\Sigma_n\to\Sigma_m\g\}$ over the set~$\mathbb{N}$. For each $\n$ define $\varphi_n\colon\Sigma_n\to\Sigma$, by $\varphi_n:=\varphi_1\circ\varphi_{21}\ldots\circ\varphi_{n(n-1)}$. Denote by $\psi\colon\mathbb{N}\map\bigcup_{n=1}^\infty\Sigma_n$ a multimap such that $\psi(n):=\Sigma_n$. Let $\varphi\colon\gr(\psi)\to\Sigma$ be a function defined by the formulae $\varphi((n,\sigma)):=\varphi_n(\sigma)$. Observe that the set $\gr(\psi)$ is directed by the relation $(n,\sigma)\succcurlyeq(m,\sigma')\overset{\text{def}}{\Longleftrightarrow}n\geqslant m\wedge\varphi_{nm}(\sigma)\succcurlyeq_m\sigma'$. It is easy to show that $\varphi$ is nondecreasing and satisfies conditions (i)--(ii). Therefore, the net $(w_{\sigma'})_{\sigma'\in\gr(\psi)}$ is finer than the initial net $(w_\sigma)_{\sigma\in\Sigma}$. Let $w:=(w^n)_{n=1}^\infty$. Then $w\in\loc$, by Lemma \ref{iso}. We claim that $w$ is a cluster point of $(w_{\sigma'})_{\sigma'\in\gr(\psi)}$ in the weak topology of the space $\loc$.\par Take $\eps>0$, $g\in\eL{1}{\Omega}^*$ and $(n,\sigma)\in\gr(\psi)$. Applying Lemma \ref{conjugate} (in a slightly informal way), one sees that there is $n_0\geqslant n$ such that $\esssupp(g)\subset\cl_\Omega(\Omega_{n_0})$. Since for each $k\in\mathbb{N}$ one has $\res{w_{\sigma'}}{\w{\Omega}_k}\xrightharpoonup[\sigma'\in\Sigma_k]{L^1(\w{\Omega}_k,E)}w^k$, we infer that
\begin{multline*}
\sup_{1\<k\<n_0}\f|\int_{\w{\Omega}_k}\f\langle\res{g}{\w{\Omega}_k}(x),w_{\varphi((n_0,\sigma'))}(x)-w^k(x)\g\rangle\,dx\g|\\\overset{(\text{ii})}{=}\sup_{1\<k\<n_0}\f|\int_{\w{\Omega}_k}\f\langle\res{g}{\w{\Omega}_k}(x),\res{w_{\varphi_{(n_0)k}(\sigma')}}{\w{\Omega}_k}(x)-w^k(x)\g\rangle\,dx\g|\xrightarrow[\sigma'\in\Sigma_{n_0}]{}0.
\end{multline*}
Taking into account that $\varphi_{(n_0)n}(\Sigma_{n_0})$ is cofinal in $\Sigma_n$ and \[\f|\f\langle g,w_{\varphi((n_0,\sigma'))}-w\g\rangle\g|\<\sum_{k=1}^{n_0}\f|\int_{\w{\Omega}_k}\f\langle\res{g}{\w{\Omega}_k}(x),w_{\varphi((n_0,\sigma'))}(x)-w^n(x)\g\rangle\,dx\g|,\] we see that there must be an index $\sigma_0\in\Sigma_{n_0}$ such that $(n_0,\sigma_0)\succcurlyeq(n,\sigma)$ and \[\f|\f\langle g,w_{\varphi((n_0,\sigma_0))}-w\g\rangle\g|<\eps.\] In other words, $w_{\varphi((n_0,\sigma_0))}\in w+g^{-1}((-\eps,\eps))$. \par Since $w$ is a cluster point of $(w_{\sigma'})_{\sigma'\in\gr(\psi)}$, it is also a cluster point of the net $(w_\sigma)_{\sigma\in\Sigma}$. Therefore, the set $A$ must be compact in the weak topology of the space $\loc$.
\end{proof}

\section{Solutions for functional integral inclusions of Volterra type}
Let $X$ be a topological space and $E$ be a Banach space. The locally convex space $C(X,E)$ endowed with the compact-open topology is complete iff $X$ is a $k$-space $($see \cite[Th.3.3.21.]{eng}$)$. If the space $X$ is $\sigma$-compact, then the space $C(X,E)$ can be metrizable in a standard manner. Therefore, the topological vector space $C(\Omega,E)$ endowed with the compact-open topology is a Fr\'echet space. It is not normable, since the local base $\f\{f\in C(X,E)\colon \sup\limits_{x\in\Omega_n}|f(x)|<\frac{1}{n}\g\}_{n=1}^\infty\,,$ generated by any exhaustion $\z{\Omega_n}$ of $\Omega$, has no bounded elements. 

\par The notion of the eponymous measure of non-compactness is laid down by the following definition:
\begin{definition}
Let $\mathbb{R}_+^{\mathbb{N}}$ be the partially ordered linear space of all positively valued sequences. Assume that $\beta\colon{\mathcal B}(E)\to\R{}_+$ is the ball measure of noncompactness on $E$, $\z{\Omega_n}$ is some exhaustion of $\Omega$ and $\tau\colon\Omega\to\R{}_+$ is a mapping. For each $N\in\mathbb{N}$ and every $L\in\R{\mathbb{N}\setminus\{1,\ldots,N-1\}}_+$ define the measure of nonequicontinuity $\nu^N_L\colon{\mathcal B}(C(\Omega,E))\to\R{\mathbb{N}\setminus\{1,\ldots,N-1\}}_+$ in the following way $\nu_L^N(M):=(\,\beta_{L_n}(M)+e_n(M))_{n=N}^\infty$, where \[\beta_{L_n}(M):=\sup_{x\in\Omega_n}e^{-L_n\tau(x)}\beta(M(x))\;\;\text{and}\;\;e_n(M):=\sup_{x\in\Omega_n}\limsup_{y\to x}\sup_{f\in M}|f(x)-f(y)|.\]  Measure $\nu_L^N$ constitutes an MNC in the sense of Definition \ref{mnc}. on the Fr\'echet space $C(\Omega,E)$ endowed with the compact-open topology $($cf. \cite[Th.5.25]{andres}$)$. Moreover, it is regular due to the fact that $\Omega$ is a locally compact space $($\cite[Th.47.1]{munkres}$)$.
\end{definition}
Let $N_F\colon C(\Omega,E)\map\loc$ be the Nemytski\v{\i} operator corresponding to the multimap $F$, i.e. \[N_F(u):=\f\{w\in L^1(\Omega,E)\colon w(x)\in F(x,u(x))\text{ a.e. on }\Omega\g\}.\] Denote by $V\colon\loc\to C(\Omega,E)$ the Volterra integral operator, given by \[V(w)(x):=\int_{\Lambda(x)}k(x,y)w(y)\,dy.\]
\par Investigation of the existence of solutions for inclusion \eqref{general} focuses, to a large degree, on the fact that the operator $N_g\circ(I\times(V\circ N_F))$ is $\nu_L^N$-condensing. Estimations, related to this argumentation, set a certain technical limitation relating to the compatibility of dimensions of the domain $\Omega$ and the Euclidean space, whose Lebesgue measurable subsets constitute the codomain of the function $\Lambda$. In order to cope with this limitation, we introduce the following 
\begin{definition}
We say that the exhaustion $\z{\Omega_n}$ is $\Lambda$-invariant, if each member $\Omega_n$ of $\z{\Omega_n}$ is invariant under $\Lambda$. Denote by ${\mathbf\Omega}(\Lambda)$ the class of $\Lambda$-invariant exhaustions of $\Omega$.
\end{definition}

\begin{example}[the class of $\Lambda$-invariant exhaustions is nonvoid]\label{ex2}\mbox{ }
\begin{itemize}
\item[(a)] Define $\Lambda\colon\Int\R{N}_+\to{\mathfrak L}(\R{N})$ by the formulae $\Lambda([x_1,\ldots,x_N]):=\prod\limits_{i=1}^N(0,x_i)$. Observe that $\Lambda$ is $\rho$-continuous. Let $\z{\Sigma_n}$ be any exhaustion of the domain $\Int\R{N}_+$. Put $\Omega_n:=\Lambda(\Sigma_n)$. Since $\Lambda$ is idempotent, one has $\Lambda(\Omega_n)=\Omega_n$ and $\bigcup_{n=1}^\infty\Omega_n=\Lambda(\Int\R{N}_+)=\Int\R{N}_+$. Moreover, $\Omega_n$ is precompact and $\Omega_n\subset\Omega_{n+1}$.
\item[(b)] Assume that $\Lambda\colon\Omega\to{\mathfrak L}(\R{N})$ is such that 
\[\begin{cases}
\forall\,x\in\Omega\;\;||\Lambda(x)||^+\<|x|\\
\forall\,x\in\Omega,\;\;B(\Lambda(x),\dist(x,\partial\Omega))\subset\Omega.
\end{cases}\]
The standard exhaustion $\z{\Omega_n}$ of $\Omega$ is given by \[\Omega_n:=\f(\R{N}\setminus D\big(\partial\Omega,n^{-1}\big)\g)\cap B(0,n)\cap\Omega.\] Clearly, $\z{\Omega_n}\in{\mathbf\Omega}(\Lambda)$.
\item[(c)] Assume that there is a point $x_0\in\R{N}$ such that 
\[\begin{cases}
\forall\,x\in\Omega\;\sup\limits_{y\in\Lambda(x)}|y-x_0|\<|x-x_0|,\\
\Lambda(\Omega)\subset\Omega.
\end{cases}\]
Let $\z{\Omega_n}$ be the exhaustion of $\Omega$ given by $\Omega_n:=B(x_0,n)\cap\Omega$. Observe that $\z{\Omega_n}\in{\mathbf\Omega}(\Lambda)$.
\end{itemize}
\end{example}
The function $\tau$ appearing in the definition of the measure of nonequicontinuity must also have some additional property enabling to demonstrate the auxiliary Lemma \ref{lemcon}. This property is described by the following
\begin{definition}\label{Xi}
We will say that an usc mapping $\tau\colon\Omega\to\R{}_+\setminus\{0\}$ is $\Lambda$-admissible, if 
\begin{equation}\label{Xi2}
\forall\,x_0\in\Omega\,\forall\,\delta>0\,\exists\,x\in B(x_0,\delta)\cap\Omega\;\;\;\sup\tau(\Lambda(x)\cap\Lambda(x_0))<\tau(x_0).
\end{equation}
Denote by ${\bm\tau}(\Lambda)$ the class of $\Lambda$-admissible mappings.
\end{definition}

\begin{example}[the class of $\Lambda$-admissible mappings is nonvoid]\label{ex3}\mbox{}
\begin{itemize}
\item[(a)] Let $\Lambda\colon\Int\R{N}_+\to\mathfrak{L}(\R{N})$ be given by $\Lambda([x_1,\ldots,x_N]):=\prod\limits_{i=1}^N(0,x_i)$. Define the function $\tau\colon\Omega\to\R{}_+$ by the formulae $\tau:=\ell\circ\Lambda$. Then $\tau$ satisfies
\begin{equation*}
\begin{cases}
|x|<|y|\Rightarrow\tau(x)<\tau(y)\\
\forall\,x\in\Omega\;\;\sup\tau(\Lambda(x))\<\tau(x)
\end{cases}
\end{equation*}
i.e., $\tau\in{\bm\tau}(\Lambda)$.
\item[(b)] Assume that $\Lambda\colon\Omega\to\mathfrak{L}(\R{N})$ and an usc mapping $\varphi\colon\R{}_+\to\R{}_+\setminus\{0\}$ satisfy
\[\begin{cases}
\forall\,x\in\Omega\;\;||\Lambda(x)||^+\<|x|\\
x<y\Rightarrow\varphi(x)<\varphi(y).
\end{cases}\]
Let $\tau\colon\Omega\to\R{}_+$ be such that $\tau(x):=\varphi(|x|)$. Clearly, $\tau\in{\bm\tau}(\Lambda)$.
\end{itemize}
\end{example}

\begin{lemma}\label{lemcon}
Let $\z{\Omega_n}$ be an exhaustion of $\Omega$ and $\tau\in{\bm\tau}(\Lambda)$. Define $\Phi\colon\R{}_+\times L^1_{\text{loc}}(\Omega,\R{}_+)\to\R{\mathbb{N}}_+$ by the formulae \[\Phi(L,\zeta)_n:=\sup_{x\in\Omega_n}e^{-L\tau(x)}\int\limits_{\Lambda(x)}e^{L\tau(y)}\zeta(y)\,dy.\] Then $\lim\limits_{L\to+\infty}\Phi(L,\zeta)_n=0$ for each fixed $(\zeta,n)\in L^1_{\text{loc}}(\Omega,\R{}_+)\times\mathbb{N}$.
\end{lemma}

\begin{proof}
Firstly observe that $e^{L\tau(\cdot)}\zeta(\cdot)\in L^1_{\text{loc}}(\Omega,R{}_+)$ and \[\limsup_{x\to x_0}\f(\,\int\limits_{\Lambda(x)}e^{L\tau(y)}\zeta(y)\,dy-\int\limits_{\Lambda(x_0)}e^{L\tau(y)}\zeta(y)\,dy\g)\<\lim_{x\to x_0}\int\limits_{\Lambda(x)\triangle\Lambda(x_0)}e^{L\tau(y)}\zeta(y)\,dy=0.\] It follows that $\Omega\ni x\mapsto e^{-L\tau(x)}\int\limits_{\Lambda(x)}e^{L\tau(y)}\zeta(y)\,dy\in\R{}_+$ is upper semicontinuous. Thus, it is sufficient to check that 
\begin{equation}\label{granica}
\lim\limits_{L\to+\infty}e^{-L\tau(x)}\int\limits_{\Lambda(x)}e^{L\tau(y)}\zeta(y)\,dy=0
\end{equation}
for every fixed $x\in\Omega$. So, let us take $x_0\in\Omega$ and $\eps>0$. Considering that $\Omega$ is open and $\Lambda$ is $\rho$-continuous, we may find $x\in\Omega$ for which
\begin{equation}\label{wciur}
\int\limits_{\Lambda(x_0)\triangle\Lambda(x)}\zeta(y)\,dy<\eps
\end{equation}
and \eqref{Xi2} is satisfied. Thus, we may estimate
\begin{align*}
0&\<\overline{\lim_{L\to\infty}}e^{-L\tau(x_0)}\int\limits_{\Lambda(x_0)}e^{L\tau(y)}\zeta(y)\,dy\\&\<\overline{\lim_{L\to\infty}}\f(\,\int\limits_{\Lambda(x_0)\setminus\Lambda(x)}e^{L(\tau(y)-\tau(x_0))}\zeta(y)\,dy+\int\limits_{\Lambda(x_0)\cap\Lambda(x)}e^{L(\tau(y)-\tau(x_0))}\zeta(y)\,dy\g)\\&\<\overline{\lim_{L\to\infty}}\f(e^{L(\sup\tau(\Lambda(x_0))-\tau(x_0))}\int\limits_{\Lambda(x_0)\triangle\Lambda(x)}\zeta(y)\,dy+e^{L(\sup\tau(\Lambda(x_0)\cap\Lambda(x))-\tau(x_0))}\int\limits_{\Lambda(x_0)\cap\Lambda(x)}\zeta(y)\,dy\g)\\&\<\int\limits_{\Lambda(x_0)\triangle\Lambda(x)}\zeta(y)\,dy+\overline{\lim_{L\to\infty}}\exp(L(\sup\tau(\Lambda(x_0)\cap\Lambda(x))-\tau(x_0)))||\zeta||_{L^1\big(\overline{\Lambda(x_0)\cap\Lambda(x)}\big)}\\&<\eps
\end{align*}
The latter implies \eqref{granica}.
\end{proof}
Let $\Delta:=\{(x,y)\in\Omega\times\Omega\colon y\in\Lambda(x)\}$. The domain $\Delta$ is nothing more than the graph $\gr(\Lambda)$ of $\Lambda$, if the latter is thought of as a set-valued map. We impose on the kernel $k\colon\Delta\to{\mathcal L}(E)$ of the Volterra integral operator $V$ the following conditions
\begin{itemize}
\item[$(\ka_1)$] $\forall\,x\in\Omega,\;\;\;k(x,\cdot)\in L^\infty(\Lambda(x),{\mathcal L}(E))$,
\item[$(\ka_2)$] $K\in C\f(\Omega,L^\infty_{\text{loc}}\f(\Omega,{\mathcal L}(E)\g)\g)$, where $K$ is induced by the mapping $k$, i.e. $K(x)(y):=k(x,y)$.
\end{itemize}

\begin{remark}
Endowed with the topology induced by a countable family of seminorms \[||f||_{L^\infty_n}:=||f(x)||_{L^\infty(\Omega_n,{\mathcal L}(E))},\;f\in L^\infty_{\text{loc}}(\Omega,{\mathcal L}(E)),\] with $\z{\Omega_n}$ being an exhaustion of $\Omega$, the space $L^\infty_{\text{loc}}(\Omega,{\mathcal L}(E))$ is locally convex and completely metrizable $($i.e., a Fr\'echet space$)$. By writing $K(x)\in L^\infty(\Omega_n,{\mathcal L}(E))$ we have in mind the trivial extension by zero from $\Lambda(x)$.
\end{remark}

\begin{remark}
Observe that the difference between the two types of continuity of operator $K$, i.e. between the assumption that $K\in C\f(\Omega,L^\infty(\Omega,{\mathcal L}(E))\g)$ and $K\in C\f(\Omega,\eL{\infty}{\Omega}\g)$, amounts to the difference between almost uniform convergence on the measure space $\Omega$ and almost uniform convergence on every compact subset of $\Omega$. 
\end{remark}

\begin{remark}
$k\in C(\Delta,{\mathcal L}(E))\Rightarrow K\in C\f(\Omega,L^\infty_{\text{loc}}\f(\Omega,{\mathcal L}(E)\g)\g)$. 
\end{remark}

Our hypotheses on the multimap $F\colon\Omega\times E\map E$ have the following form:
\begin{itemize}
\item[$(\F_1)$] for every $(x,u)\in \Omega\times E$ the set $F(x,u)$ is nonempty closed and convex,
\item[$(\F_2)$] the map $F(\cdot,u)$ has a strongly measurable selection for every $u\in E$,
\item[$(\F_3)$] the map $F(x,\cdot)$ is upper hemicontinuous for a.a. $x\in\Omega$,
\item[$(\F_4)$] there exists $b\in L^1_{\text{loc}}(\Omega)$ such that \[||F(x,u)||^+\<b(x)(1+|u|)\;\text{ a.e. on }\Omega, \text{ for all }u\in E,\]
\item[$(\F_5)$] there is a function $\eta\in L^1_{\text{loc}}(\Omega)$ such that for all bounded $M$ in $E$ and for a.a. $x\in\Omega$ the inequality holds \[\beta(F(x,M))\<\eta(x)\beta(M).\]
%\item[$(\F_5)$] for every closed separable linear subspace $E_0$ of $E$ and every compact subset $K$ of the set $\Omega$ the map $\res{F}{K\times E_0}(x,\cdot)\cap E_0$ is quasicompact for a.a. $x\in K$.
\end{itemize}

Regularity of the Niemytski\v{\i} operator $N_F$, necessary from our point of view, poses a consequence of the following
\begin{lemma}\label{nem}
Assume $(\mathbb{E})$. Under conditions $(\F_1)$-$(\F_5)$ the set--valued Nemytski\v{\i} operator $N_F\colon C(\Omega,E)\map\loc$ is a strict weakly upper semicontinuous set-valued map with weakly compact convex values.
\end{lemma}

\begin{proof}
Assume that $\z{\Omega_n}$ is an exhaustion of $\Omega$. Let $u\in C(\Omega,E)$ and $u_n:=\res{u}{\w{\Omega}_n}$. There is a sequence $(u_k^n)_{k=1}^\infty$ of $\ell$-simple functions, which converges to $u_n$ in the norm of $L^\infty(\w{\Omega}_n,E)$. In particular, for each $\n$ there exists a sequence $(m_k^n)_{n=1}^\infty$ such that $u^n_{m_k^n}(x)\xrightarrow[k\to\infty]{E}u_n(x)$ a.e. on $\w{\Omega}_n$. Accordingly to the assumption $(\F_2)$ we can indicate a strongly measurable map $w_{m_k^n}^n\colon\w{\Omega}_n\to E$ such that $w_{m_k^n}^n(x)\in F(x,u_{m_k^n}^n(x))$ a.e. on $\w{\Omega}_n$. Since \[\sup_\K|w_{m_k^n}^n(x)|\<\sup_\K||F(x,u_{m_k^n}^n(x))||^+\<b(x)\f(1+\sup_\K||u_{m_k^n}^n||_{L^\infty(\w{\Omega}_n,E)}\g)\;\;\text{a.e. on }\w{\Omega}_n\] and the slice $\big\{w_{m_k^n}^n(x)\big\}_{k=1}^\infty$ is relatively weakly compact in $E$ as a subset of $F\f(x,\{\overline{u_{m_k^n}^n(x)}\}_{k=1}^\infty\g)$, it follows, from Theorem \ref{ulger}, that $w_{m_k^n}^n\xrightharpoonup[k\to\infty]{L^1(\w{\Omega}_n,E)}w^n$, up to a subsequence. In view of the convergence theorem (\cite[Corollary 1]{pietkun}), $w^n(x)\in F(x,u_n(x))$ for $x\in\w{\Omega}_n\setminus I_n$ with $\ell(I_n)=0$. Put $w:=(w^n)_{n=1}^\infty$. By Lemma \ref{iso}, $w\in\loc$. Observe that $\ell\big(\bigcup\limits_{n=1}^\infty I_n\big)=0$ and $w(x)\in F(x,u(x))$ for $x\in\bigcup\limits_{n=1}^\infty\w{\Omega}_n\setminus I_n=\Omega\setminus\bigcup\limits_{n=1}^\infty I_n$. In other words, $w\in N_F(u)$.\par Assume that $u_n\xrightarrow[n\to\infty]{C(\Omega,E)}u$ and $w_n\in N_F(u_n)$ for $\n$. Clearly, the set $\x{w_n}$ is locally integrably bounded and the the cross-section $\x{w_n(x)}$ is relatively weakly compact in $E$ for a.a. $x\in\Omega$. Therefore, $\x{w_n}$ must be relatively weakly compact in $\loc$, by virtue of Theorem \ref{extension}. Since $\loc$ is metrizable locally convex space, it is weakly angelic (see \cite[Theorem 11]{orihuela}). Thus, $\x{w_n}$ is relatively sequentially compact in the weak topology. We may assume, passing to a subsequence if necessary, that $w_n\xrightharpoonup[n\to\infty]{\loc}w$. Since for each $\K$
\[\begin{cases}
\res{u_n}{\Omega_k}(x)\xrightarrow[n\to\infty]{E}\res{u}{\Omega_k}(x),&\text{for }x\in\Omega_k\\
\res{w_n}{\Omega_k}\xrightharpoonup[n\to\infty]{L^1(\Omega_k,E)}\res{w}{\Omega_k}\\
\res{w_n}{\Omega_k}(x)\in F\big(x,\res{u_n}{\Omega_k}(x)\big),&\text{a.e. on }\Omega_k,
\end{cases}\]
it follows that $\res{w}{\Omega_k}(x)\in F\big(x,\res{u}{\Omega_k}(x)\big)$ a.e. on $\Omega_k$ for every $\K$, by the convergence theorem (\cite[Corollary 1]{pietkun}). Eventually $w\in N_F(u)$, which means that the Nemytski\v{\i} operator $N_F$ is a weakly upper semicontinuous operator with weakly compact values.
\end{proof}
For the purpose of showing that $V\circ N_F$ is upper semicontinuous we have to prove 
\begin{lemma}
Assume that ${\mathbf\Omega}(\Lambda)\neq\varnothing$. Under conditions $(\ka_1)$-$(\ka_2)$ the operator $V$ is continuous.
\end{lemma}

\begin{proof}
Let $\z{\Omega_n}\in{\mathbf\Omega}(\Lambda)$. Operator $V$ is well-defined. Let $x_n\xrightarrow[n\to\infty]{\Omega}x$. Then
\begin{align*}
|V(w)(x_n)-V(w)(x)|&=\f|\,\int\limits_{\Lambda(x_n)}k(x_n,y)w(y)\,dy-\int\limits_{\Lambda(x)}k(x,y)w(y)\,dy\g|\\&\<\int\limits_{\cl_\Omega(\Omega_k)}\f|k(x_n,y)w(y){\bf 1}_{\Lambda(x_n)}-k(x,y)w(y){\bf 1}_{\Lambda(x)}\g|\,dy\\&\<\int\limits_{\cl_\Omega(\Omega_k)}||k(x_n,y)-k(x,y)||_{{\mathcal L}(E)}|w(y)|\,dy\\&+\int\limits_{\cl_\Omega(\Omega_k)}||k(x,y)||_{{\mathcal L}(E)}|w(y)|{\bf 1}_{\Lambda(x_n)\triangle\Lambda(x)}(y)\,dy\\&\<||K(x_n)-K(x)||_{L^\infty(\cl_\Omega(\Omega_k),{\mathcal L}(E))}||w||_{L^1(\cl_\Omega(\Omega_k),E)}\\&+||K(x)||_{L^\infty(\cl_\Omega(\Omega_k),{\mathcal L}(E))}\int\limits_{\Lambda(x_n)\triangle\Lambda(x)}|w(y)|\,dy=:\alpha_n
\end{align*}
with $\Lambda\f(\overline{\x{x_n}}\g)\subset\Omega_k$ for some $k\in\mathbb{N}$. Since $K(x_n)\xrightarrow[n\to\infty]{L^\infty\big(\cl_\Omega(\Omega_k),{\mathcal L}(E)\big)}K(x)$ and $\Lambda$ is continuous, i.e. $\ell(\Lambda(x_n)\triangle\Lambda(x))\xrightarrow[n\to\infty]{}0$, we see that $\alpha_n\xrightarrow[n\to\infty]{}0$. Hence, $V(w)\in C(\Omega,E)$.\par Suppose that $w_k\xrightarrow[k\to\infty]{\loc}w$. Fix an arbitrary $n\in\mathbb{N}$. Then $\Lambda(\Omega_n)\subset\Omega_n$ and
\begin{align*}
\sup_{x\in\cl_\Omega(\Omega_n)}|V(w_k)(x)-V(w)(x)|&=\sup_{x\in\Omega_n}|V(w_k)(x)-V(w)(x)|\\&=\sup_{x\in\Omega_n}\f|\;\int\limits_{\Lambda(x)}k(x,y)w_k(y)\,dy-\int\limits_{\Lambda(x)}k(x,y)w(y)\,dy\g|\\&\<\sup_{x\in\Omega_n}\,\int\limits_{\Lambda(x)}||k(x,y)||_{{\mathcal L}(E)}|w_k(y)-w(y)|\,dy\\&\<\sup_{x\in\Omega_n}\;\int\limits_{\cl_\Omega(\Omega_n)}||k(x,y)||_{{\mathcal L}(E)}|w_k(y)-w(y)|\,dy\\&\<\sup_{x\in\Omega_n}||K(x)||_{L^\infty\big(\cl_\Omega(\Omega_n),{\mathcal L}(E)\big)}||w_k-w||_{L^1\big(\cl_\Omega(\Omega_n),E\big)}.
\end{align*}
Since $K\in C\f(\Omega,L^\infty\big(\cl_\Omega(\Omega_n),{\mathcal L}(E)\big)\g)$ is continuous, $\sup\limits_{x\in\Omega_n}||K(x)||_{L^\infty\big(\cl_\Omega(\Omega_n),{\mathcal L}(E)\big)}<+\infty$. Thus $V(w_k)\xrightarrow[k\to\infty]{C(\Omega,E)}V(w)$, which means that the integral operator $V\colon\loc\to C(\Omega,E)$ is continuous.
\end{proof}
The hereunder multivalued Volterra integral equation with inhomogeneity presents a version of inclusion \eqref{general}, to which the first result regarding the existence of solutions is devoted.
\begin{equation}\label{inclusion2}
u(x)\in g(x,u(x))+\int\limits_{\Lambda(x)}k(x,y)F(y,u(y))\,dy,\;\;x\in\Omega
\end{equation}
Put $||\cdot||_n:=||\cdot||_{C\f(\cl_\Omega(\Omega_n),E\g)}$ and \[\phi:=\big\{\varphi\colon\R{}_+\to\R{}_+\colon\varphi\text{ is nondecreasing usc and satisfies \eqref{concave}}\big\}.\]
\begin{equation}\label{concave}
\forall\,x>0\;\;\lim\limits_{n\to\infty}\varphi^n(x)=0
\end{equation}

Our hypotheses on the mapping $g\colon\Omega\times E\to E$ are as follows:
\begin{itemize}
\item[$(\gie_1)$] $g$ is uniformly continuous on bounded subsets of $\Omega\times E$,
\item[$(\gie_2)$] there exists a concave $\varphi\in\phi$ satisfying 
\begin{equation}\label{10}
\limsup\limits_{x\to 0^+}\,\frac{\varphi(x)}{x}<1
\end{equation}
for which \[|g(x,u)-g(x,w)|\<\varphi(|u-w|)\] for all $u,w\in E$ and $x\in\Omega$.
\end{itemize}

\begin{theorem}\label{ex1}
Assume ${\bm\Omega}(\Lambda)\neq\varnothing$ and ${\bm\tau}(\Lambda)\neq\varnothing$. Let $(\mathbb{E})$ be satisfied. Suppose that hypotheses $(\ka_1)$-$(\ka_2)$, $(\gie_1)$-$(\gie_2)$ and $(\F_1)$-$(\F_5)$ hold, together with 
\begin{equation}\label{brzeg}
\liminf_{n\to\infty}\,\f(a_n-\varphi(a_n)-||g(\cdot,0)||_n\g)>0
\end{equation}
for some $\z{a_n}\in\R{\mathbb{N}}_+$. Then the Volterra integral inclusion \eqref{inclusion2} has at leat one continuous solution.
\end{theorem}

\begin{remark}
If $\Omega$ is bounded, then $C(\Omega,E)$ with the usual supremum norm is a Banach space. In these circumstances condition \eqref{brzeg} amounts to the existence of an $r>0$, which satisfies \[\varphi(r)+\sup_{x\in\Omega}|g(x,0)|<r.\] In this form, it resembles very much condition $(3.20)$ in \cite[Lemma 3.5]{khandani}.
\end{remark}

\begin{example}
Fix $k\in(0,1)$.
\begin{itemize}
\item[(i)] Let $\varphi\colon\R{}_+\to\R{}_+$ be given by $\varphi(x):=kx$. 
\item[(ii)] Define $\varphi\colon\R{}_+\to\R{}_+$ by $\varphi(x):=\arctan(kx)$.
\end{itemize}
In both cases $\varphi$ is concave, belongs to the class $\phi$ and satisfies \eqref{10}.
\end{example}

\begin{proof}
Take $\z{\Omega_n}\in{\bm\Omega}(\Lambda)$ and $\tau\in{\bm\tau}(\Lambda)$. From Lemma \ref{lemcon} and assumption \eqref{brzeg} follows the existence of $N\in\mathbb{N}$ and $L\in\R{\mathbb{N}\setminus\{1,\ldots,N-1\}}_+$, for which the following inequality is satisfied:
\begin{equation}\label{fi}
\Phi(L_n,b)_n\<\frac{a_n-\varphi(a_n)-||g(\cdot,0)||_n}{\sup\limits_{x\in\Omega_n}||K(x)||_{L^\infty\big(\cl_\Omega(\Omega_n),\mathcal{L}(E)\big)}(1+a_n)}.
\end{equation}
We may assume w.l.o.g. that $(L_n)_{n=N}^\infty$ in nondecreasing. Let $H\colon C(\Omega,E)\map C(\Omega,E)$ be given by the formula $H:=N_g+V\circ N_F$. We will show the non-emptiness of $\fix(H)$ with the aid of a routine renorming technique. Namely, let \[||u||_{L_n}:=\sup\limits_{x\in\Omega_n}e^{-L_n\tau(x)}|u(x)|\;\text{ for }u\in C(\Omega,E).\] Clearly, the family $\{||\cdot||_{L_n}\}_{n=N}^\infty$ generates the same compact-open topology on $C(\Omega,E)$, since $||\cdot||_{L_n}\<||\cdot||_n\<e^{L_n\sup\tau(\Omega_n)}||\cdot||_{L_n}$.\par Put 
\begin{equation}\label{X}
{\mathcal X}:=\bigcap_{n=N}^\infty\f\{u\in C(\Omega,E)\colon ||u||_{L_n}\<a_n\g\}.
\end{equation}
It is easy to see that ${\mathcal X}$ forms closed and convex subset of the space $C(\Omega,E)$. Obviously, ${\mathcal X}$ is topologically bounded, since it is bounded with respect to each seminorm $||\cdot||_{L_n}$. We claim that ${\mathcal X}$ is invariant under the operator $H$. Fix $v\in H({\mathcal X})$. Then $v=N_g(u)+V(w)$ for some $w\in N_F(u)$ and $u\in{\mathcal X}$. One easily sees that
\begin{align*}
|v(x)|&\<|g(x,u(x))|+\int\limits_{\Lambda(x)}||k(x,y)||_{{\mathcal L}(E)}|w(y)|\,dy\\&\<\varphi(|u(x)|)+||g(\cdot,0)||_n+\int\limits_{\Lambda(x)}||k(x,y)||_{{\mathcal L}(E)}b(y)(1+|u(y)|)\,dy.
\end{align*}
Concavity of $\varphi$ entails $\lambdaup\varphi(x)\<\varphi(\lambdaup x)$ for $\lambdaup\in(0,1)$ and $x\in\R{}_+$. Hence, for each $n\geqslant N$ one has
\begin{equation}\label{invariant}
\begin{split}
||v||_{L_n}&\<\varphi(||u||_{L_n})+||g(\cdot,0)||_n+\sup_{x\in\Omega_n}||K(x)||_{L^\infty_n}\Phi(L_n,b)_n\f(1+||u||_{L_n}\g)\\&\<\varphi(a_n)+||g(\cdot,0)||_n+\sup_{x\in\Omega_n}||K(x)||_{L^\infty_n}\Phi(L_n,b)_n(1+a_n)\<a_n,
\end{split}
\end{equation}
where $L^\infty_n:=L^\infty\big(\cl_\Omega(\Omega_n),\mathcal{L}(E)\big)$, by \eqref{fi}.
\par Now, we will show that $H\colon{\mathcal X}\map{\mathcal X}$ is acyclic. To this aim assume that $u_n\xrightarrow[n\to\infty]{C(\Omega,E)}u$, $v_n=N_g(u_n)+V(w_n)$ and $w_n\in N_F(u_n)$ for $\n$. By virtue of \cite[Th.3.12.]{kunze}, the following estimate remains in force 
\begin{align*}
\sup_{x\in\Omega_k}\beta\f(\x{V(w_n)(x)}\g)&=\sup_{x\in\Omega_k}\beta\f(\f\{\,\int\limits_{\Lambda(x)}\!\!k(x,y)w_n(y)\,dy\g\}_{n=1}^\infty\g)\<4\sup_{x\in\Omega_k}\overline{\int\limits_{\Lambda(x)}}\!\beta\f(k(x,y)\x{w_n(y)}\g)dy\\&\<4\sup_{x\in\Omega_k}\overline{\int\limits_{\Lambda(x)}}||k(x,y)||_{{\mathcal L}(E)}\,\beta\f(F\f(y,\x{u_n(y)}\g)\g)\,dy.
\end{align*}
Since $\beta\f(F\f(x,\x{u_n(x)}\g)\g)\<\eta(x)\beta\f(\x{u_n(x)}\g)=0$ for a.a. $x\in\Omega$, we conclude that $\sup\limits_{x\in\Omega_k}\beta\f(\x{V(w_n)(x)}\g)=0$ for $\K$.
On the other hand, we have
\begin{equation}\label{V}
\begin{split}
\sup_{x\in\Omega_k}\overline{\lim_{z\to x}}\sup_\n|V(w_n)(x)-V(w_n)(z)|&=\sup_{x\in\Omega_k}\overline{\lim_{z\to x}}\sup_\n\f|\,\int\limits_{\Lambda(x)}k(x,y)w_n(y)\,dy-\int\limits_{\Lambda(z)}k(z,y)w_n(y)\,dy\g|\\&\<\sup_{x\in\Omega_k}\overline{\lim_{z\to x}}\sup_\n\int\limits_{\Omega_k}||k(x,y)-k(z,y)||_{{\mathcal L}(E)}|w_n(y)|{\bf 1}_{\Lambda(x)}\,dy\\&+\sup_{x\in\Omega_k}\overline{\lim_{z\to x}}\sup_\n\int\limits_{\Omega_k}||k(z,y)||_{{\mathcal L}(E)}|w_n(y)|\f|{\bf 1}_{\Lambda(x)}-{\bf 1}_{\Lambda(z)}\g|\,dy\\&\<R\sup_{x\in\Omega_k}\overline{\lim_{z\to x}}||K(x)-K(z)||_{L^\infty\big(\overline{\Omega}_k,{\mathcal L}(E)\big)}||b||_{L^1\big(\overline{\Omega}_k\big)}\\&+R\sup_{x\in\Omega_k}\overline{\lim_{z\to x}}||K(z)||_{L^\infty\big(\overline{\Omega}_k,{\mathcal L}(E)\big)}\int\limits_{\Lambda(x)\triangle\Lambda(z)}b(y)\,dy\\&=R\sup_{x\in\Omega_k}||K(x)||_{L^\infty\big(\overline{\Omega}_k,{\mathcal L}(E)\big)}\overline{\lim_{z\to x}}\int\limits_{\Lambda(x)\triangle\Lambda(z)}b(y)\,dy=0,
\end{split}
\end{equation}
where $R:=1+\sup\limits_\n||u_n||_k$. Therefore, the family $\f\{N_g(u_n)+V(w_n)\g\}_{n=1}^\infty$ forms a relatively compact subset of $C(\Omega,E)$. Consequently, there exists $\z{m_n}\in\mathbb{N}^\mathbb{N}$ such that $v_{m_n}=N_g(u_{m_n})+V(w_{m_n})\xrightarrow[n\to\infty]{C(\Omega,E)}z$. In view of Lemma \ref{nem}, $w_{m_n}\xrightharpoonup[n\to\infty]{\loc}w\in N_F(u)$, up to a subsequence. Taking into account that $V\in C\f((\loc,w),(C(\Omega,E),w)\g)$ and $N_g\in C(C(\Omega,E),C(\Omega,E))$, one may deduce $v_{m_n}=N_g(u_{m_n})+V(w_{m_n})\xrightharpoonup[n\to\infty]{C(\Omega,E)}N_g(u)+V(w)$. Eventually, $v_{m_n}\xrightarrow[n\to\infty]{C(\Omega,E)}N_g(u)+V(w)\in H(u)$. Summing up, $H$ is an upper semicontinuous operator with compact and convex values. \par Put $r_n:=\exp\big(L_n\sup\tau(\Omega_n)\big)\cdot a_n$ for $n\geqslant N$. Upper semicontinuity of $\varphi$ and assumption \eqref{10} imply \[\forall\,n\geqslant N\;\exists\,\lambdaup_n\in(0,1)\;\;\sup_{0<x\<r_n}\frac{\varphi(x)}{x}\<\lambdaup_n.\] Whence 
\begin{equation}\label{r_n}
\forall\,n\geqslant N\;\exists\,k_n\in(0,1)\;\forall\,x\in(0,r_n]\;\;\varphi(x)+k_nx<x.
\end{equation}
In view of Lemma \ref{lemcon} there exists $\hat{L}\in\R{\mathbb{N}\setminus\{1,\ldots,N-1\}}_+$ such that \[4\sup\limits_{x\in\Omega_n}||K(x)||_{L^\infty_n}\Phi(\hat{L}_n,\eta)_n<k_n\] for $n\geqslant N$. Let $\psi_n\colon\R{}_+\to\R{}_+$ be such that $\psi_n(x):=\varphi(x)+k_nx$. Notice that $\psi_n$ is concave and for all $x\in\R{}_+$ one has $\psi_n(x)-x\<0$, by \eqref{r_n} and \eqref{concave} (actually, $\psi_n(x)<x$ for $x>0$). Define $f\colon\R{\mathbb{N}\setminus\{1,\ldots,N-1\}}_+\times\R{\mathbb{N}\setminus\{1,\ldots,N-1\}}_+\to\big(\R{\mathbb{N}\setminus\{1,\ldots,N-1\}},\geqslant\big)$ by the formulae \[f\f((x_n)_{n=N}^\infty,(y_n)_{n=N}^\infty\g)=(\psi_n(y_n)-x_n)_{n=N}^\infty.\] Clearly, $f\in\Phi$. Our next goal is to show that the operator $H\colon{\mathcal X}\map {\mathcal X}$ meets the assumption \eqref{cond} of Theorem~\ref{fixed} in the context of some measure of nonequicontinuity and the mapping $f$.\par Suppose that $M\subset{\mathcal X}$ is not relatively compact. Note that 
\begin{equation}\label{M(x)}
\sup_{x\in\Omega_n}||M(x)||^+\<\exp\big(L_n\sup\tau(\Omega_n)\big)||M||_{L_n}^+\<r_n
\end{equation}
for each $n\geqslant N$. Since $\beta(g(x,M(x)))\<\varphi\f(\eps+\beta(M(x))\g)$ for every $\eps>0$ one has \[\beta(g(x,M(x)))\<\limsup_{\eps\to 0^+}\varphi(\eps+\beta(M(x)))\<\varphi(\beta(M(x)))\] for $x\in\Omega$. On the other hand, for each $x\in\Omega_n$ and $\n$ one has \[\beta\f(N_F(M)(x)\g)\<\beta(F(x,M(x)))\<\eta(x)\beta(M(x))\<e^{\hat{L}_n\tau(x)}\eta(x)\beta_{\hat{L}_n}(M).\] Taking into account above findings, one sees that
\begin{equation}\label{szacowanie}
\begin{split}
\beta_{\hat{L}_n}&(H(M))\\&\<\sup_{x\in\Omega_n}e^{-\hat{L}_n\tau(x)}\f(\varphi\f(\beta(M(x))\g)+\beta\f(\f\{\,\int\limits_{\Lambda(x)}k(x,y)w(y)\,dy\colon w\in N_F(M)\g\}\g)\g)\\&\<\sup_{x\in\Omega_n}e^{-\hat{L}_n\tau(x)}\varphi(\beta(M(x)))+4\sup_{x\in\Omega_n}e^{-\hat{L}_n\tau(x)}\overline{\int\limits_{\Lambda(x)}}||k(x,y)||_{{\mathcal L}(E)}\,\beta\f(N_F(M)(y)\g)\,dy\\&\<\varphi\f(\beta_{\hat{L}_n}(M)\g)+4\sup_{x\in\Omega_n}||K(x)||_{L^\infty_n}\sup_{x\in\Omega_n}e^{-\hat{L}_n\tau(x)}\int\limits_{\Lambda(x)}e^{\hat{L}_n\tau(y)}\eta(y)\,dy\,\beta_{\hat{L}_n}(M)\\&=\varphi\f(\beta_{\hat{L}_n}(M)\g)+4\sup_{x\in\Omega_n}||K(x)||_{L^\infty_n}\Phi(\hat{L}_n,\eta)_n\,\beta_{\hat{L}_n}(M)\<\varphi\f(\beta_{\hat{L}_n}(M)\g)+k_n\beta_{\hat{L}_n}(M)\\&=\psi_n\f(\beta_{\hat{L}_n}(M)\g)
\end{split}
\end{equation}
for $n\geqslant N$. In the above estimation we utilized the fact that $\z{\Omega_n}$ is $\Lambda$-invariant. Theorem \cite[Th.3.12.]{kunze} was also applied. \par Observe that \[\sup_{x\in\Omega_n}\overline{\lim_{z\to x}}\sup_{u\in M}|g(x,u(z))-g(z,u(z))|=0,\] by $(\gie_1)$. Taking into account that $\limsup\limits_{z\to x}\varphi(\psi(z))\<\varphi\f(\limsup\limits_{z\to x}\psi(z)\g)$ for any $\psi\colon\R{}_+\to\R{}_+$, one may estimate
\begin{align*}
\sup_{x\in\Omega_n}\overline{\lim_{z\to x}}\sup_{u\in M}&\,|N_g(u)(x)-N_g(u)(z)|\\&\<\sup_{x\in\Omega_n}\overline{\lim_{z\to x}}\sup_{u\in M}\f(|g(x,u(x))-g(x,u(z))|+|g(x,u(z))-g(z,u(z))|\g)\\&\<\sup_{x\in\Omega_n}\overline{\lim_{z\to x}}\sup_{u\in M}\varphi(|u(x)-u(z)|)+\sup_{x\in\Omega_n}\overline{\lim_{z\to x}}\sup_{u\in M}|g(x,u(z))-g(z,u(z))|\\&\<\sup_{x\in\Omega_n}\varphi\f(\overline{\lim_{z\to x}}\sup_{u\in M}|u(x)-u(z)|\g)\<\varphi(e_n(M))
\end{align*}
At the same time
\begin{align*}
\sup_{x\in\Omega_n}\overline{\lim_{z\to x}}\sup_{w\in N_F(M)}|V(w)(x)-V(w)(z)|&\<\f(1+||M||_n^+\g)\sup_{x\in\Omega_n}\overline{\lim_{z\to x}}||K(x)-K(z)||_{L^\infty_n}||b||_{L^1\big(\overline{\Omega}_n\big)}\\&+\f(1+||M||_n^+\g)\sup_{x\in\Omega_n}||K(x)||_{L^\infty_n}\overline{\lim_{z\to x}}\!\int\limits_{\Lambda(x)\triangle\Lambda(z)}\!\!\!b(y)\,dy\\&=0,
\end{align*}
by \eqref{V}. Therefore
\begin{equation}\label{szacowanie2}
\begin{split}
e_n(H(M))&=\sup_{x\in\Omega_n}\overline{\lim_{z\to x}}\sup_{v\in H(M)}|v(x)-v(z)|\\&\<\sup_{x\in\Omega_n}\overline{\lim_{z\to x}}\sup_{u\in M}|N_g(u)(x)-N_g(u)(z)|+\sup_{x\in\Omega_n}\overline{\lim_{z\to x}}\sup_{w\in N_F(M)}|V(w)(x)-V(w)(z)|\\&\<\varphi(e_n(M)).
\end{split}
\end{equation}
for $n\geqslant N$. Taking into consideration \eqref{szacowanie} and \eqref{szacowanie2} we may sum up
\begin{align*}
\frac{1}{2}\nu^N_{\hat{L}}(H(M))_n&=\frac{1}{2}\beta_{\hat{L}_n}(H(M))+\frac{1}{2}e_n(H(M))\<\frac{1}{2}\psi_n\f(\beta_{\hat{L}_n}(M)\g)+\frac{1}{2}\varphi(e_n(M))\\&\<\frac{1}{2}\psi_n\f(\beta_{\hat{L}_n}(M)\g)+\frac{1}{2}\psi_n(e_n(M))\<\psi_n\f(\frac{1}{2}\beta_{\hat{L}_n}(M)+\frac{1}{2}e_n(M)\g).
\end{align*}
for $n\geqslant N$. Denoting $\w{\nu}^N_{\hat{L}}:=\frac{1}{2}\nu^N_{\hat{L}}$ one may rewrite the latter inequality in the following form \[\w{\nu}^N_{\hat{L}}(H(M))_n\<\psi_n\f(\w{\nu}^N_{\hat{L}}(M)_n\g).\] \par Since $\overline{M}$ is noncompact, there must be an index $n_0\geqslant N$ such that $\w{\nu}^N_{\hat{L}}(M)_{n_0}>0$. So, we are dealing with the alternative: $\beta_{\hat{L}_{n_0}}(M)>0$ or $e_{n_0}(M)>0$.In both cases, it follows from \eqref{szacowanie} and \eqref{szacowanie2} respectively that $\w{\nu}^N_{\hat{L}}(H(M))_{n_0}<\w{\nu}^N_{\hat{L}}(M)_{n_0}$. Thus, $f\big(\w{\nu}^N_{\hat{L}}(H(M)),\w{\nu}^N_{\hat{L}}(M)\big)\neq 0$. The latter means that $f\big(\w{\nu}^N_{\hat{L}}(H(M)),\w{\nu}^N_{\hat{L}}(M)\big)\in\R{\mathbb{N}\setminus\{1,\ldots,N-1\}}_+\setminus\{0\}$, i.e. assumption \eqref{cond} of Theorem \ref{fixed} is met. In connection with that, $\fix(H)$ must be nonempty. Consequently, the integral inclusion \eqref{inclusion2} possesses a continuous solution.
\end{proof}

\begin{corollary}
Let $(\mathbb{E})$ be satisfied. Assume ${\bm\Omega}(\Lambda)\neq\varnothing$ and there exists a continuous $\tau\in{\bm\tau}(\Lambda)$. Suppose there exists a nondecreasing positively homogeneous usc at zero function $\theta\colon\R{}_+\to\R{}_+$ with $\theta(0)=0$ and a concave function $\varphi\in\phi$ satisfying \eqref{10}, for which
\begin{equation}\label{gie}
|g(x,u)-g(y,w)|\<\theta(|x-y|)+\varphi(|u-w|)\;\text{ for all }(x,u), (y,w)\in\Omega\times E.
\end{equation}
Assume further that hypotheses $(\ka_1)$-$(\ka_2)$ and $(\F_1)$-$(\F_5)$ hold. Then the Volterra integral inclusion \eqref{inclusion2} has at leat one continuous solution.
\end{corollary}

\begin{proof}
Notice that \eqref{gie} entails $(\gie_1)$--$(\gie_2)$. Fix any $r>0$. Clearly, $\inf\tau(\Omega_n)>0$ for each $n\in\mathbb{N}$, by continuity of $\tau$. Since $\theta$ is usc at zero and \[\sup_{x\in\Omega_n}e^{-L\tau(x)}|x|\<e^{-L\inf\tau(\Omega_n)}||\Omega_n||^+\xrightarrow[L\to+\infty]{}0,\] we may choose in accordance with the latter and Lemma \ref{lemcon} a sequence $L\in\R{\mathbb{N}}_+$ for which \[r-\varphi(r)-\theta\big(||\Omega_n||^+_{L_n}\big)-\sup_{x\in\Omega_n}||K(x)||_{L^\infty_n}\Phi(L_n,b)_n(1+r)\geqslant 0.\] Modify definition \eqref{X} in the following way \[{\mathcal X}:=\bigcap_{n=1}^\infty\f\{u\in C(\Omega,E)\colon ||u||_{L_n}\<r\g\}.\] In connection with the above, inequality \eqref{invariant} will gain the form
\begin{equation}
\begin{split}
||v||_{L_n}&\<\varphi(||u||_{L_n})+\sup_{x\in\Omega_n}e^{-L_n\tau(x)}\theta(|x|)+\sup_{x\in\Omega_n}||K(x)||_{L^\infty_n}\Phi(L_n,b)_n\f(1+||u||_{L_n}\g)\\&\<\varphi(r)+\theta(||\Omega_n||^+_{L_n})+\sup_{x\in\Omega_n}||K(x)||_{L^\infty_n}\Phi(L_n,b)_n(1+r)\<r.
\end{split}
\end{equation}
Consequently, the set ${\mathcal X}$ is invariant under the operator $H$. In the context of proof of Theorem \ref{ex1} it is clear that the integral inclusion \eqref{inclusion2} possesses a continuous solution.
\end{proof}

\begin{corollary}
Assume ${\bm\Omega}(\Lambda)\neq\varnothing$ and ${\bm\tau}(\Lambda)\neq\varnothing$. Let $(\mathbb{E})$ be satisfied. Suppose that hypotheses $(\ka_1)$-$(\ka_2)$  and $(\F_1)$-$(\F_5)$ hold. If assumptions $(\gie_1)$-$(\gie_2)$ are satisfied with the proviso that $\varphi\in\phi$ is given by $\varphi(x):=kx$ for some $k\in(0,1)$ and $R:=\sup\limits_{x\in\Omega}|g(x,0)|<\infty$, then the solution set of the Volterra integral inclusion \eqref{inclusion2} is nonempty and compact in the compact-open topology of $C(\Omega,E)$.
\end{corollary}

\begin{proof}
Take $\z{\Omega_n}\in{\bm\Omega}(\Lambda)$ and $\tau\in{\bm\tau}(\Lambda)$. Put $a_n:=n$. Clearly, \[\liminf_{n\to\infty}(a_n-\varphi(a_n)-||g(\cdot,0)||_n)=\lim_{n\to\infty}(1-k)n-R=+\infty>0.\] From Lemma \ref{lemcon} follows the existence of $L\in\R{\mathbb{N}\setminus\{1,\ldots,N-1\}}_+$, for which the following inequality is satisfied:
\begin{equation}\label{fi2}
\Phi(L_n,b)_n\<\frac{n-kn-R}{\sup\limits_{x\in\Omega_n}||K(x)||_{L^\infty_n}(1+n)}.
\end{equation}
Consider ${\mathcal X}$ given by \eqref{X}. Denote by ${\mathcal S}$ the solution set of the problem \eqref{inclusion2}. We show that $H({\mathcal X})\subset{\mathcal X}$ and at the same time ${\mathcal S}\subset{\mathcal X}$. To this aim fix $v=N_g(u)+V(w)\in N_g(u)+V(N_F(u))\subset H({\mathcal X})$ and $\hat{u}\in{\mathcal S}$. Since $|g(x,u)|\<k|u|+|g(x,0)|$, we arrive at
\begin{align*}
||v||_{L_n}&\<k||u||_{L_n}+||g(\cdot,0)||_n+\sup_{x\in\Omega_n}||K(x)||_{L^\infty_n}\Phi(L_n,b)_n\f(1+||u||_{L_n}\g)\\&\<kn+R+\sup_{x\in\Omega_n}||K(x)||_{L^\infty_n}\Phi(L_n,b)_n(1+n)\<n,
\end{align*}
which means that ${\mathcal X}$ is $H$-invariant. On the other hand, from
\begin{align*}
||\hat{u}||_{L_n}&\<k||\hat{u}||_{L_n}+R+\sup_{x\in\Omega_n}||K(x)||_{L^\infty_n}\Phi(L_n,b)_n\f(1+||\hat{u}||_{L_n}\g),
\end{align*}
it follows \[||\hat{u}||_{L_n}\<\frac{R+\sup\limits_{x\in\Omega_n}||K(x)||_{L^\infty_n}\Phi(L_n,b)_n}{1-k-\sup\limits_{x\in\Omega_n}||K(x)||_{L^\infty_n}\Phi(L_n,b)_n}\<n,\] by \eqref{fi2} and inclusion ${\mathcal S}\subset{\mathcal X}$ follows.\par The rest of the proof proceeds analogously to the proof of Theorem \ref{ex1}. In particular, the fixed point set $\fix(H)$ is compact in the compact-open topology of the space $C(\Omega,E)$, in view of Theorem \ref{fixed}. Since ${\mathcal S}=\fix(H)$, the solution set of \eqref{inclusion2} must be also compact.
\end{proof}
The successive existence theorem applies to the following generalization of the integral inclusion \eqref{inclusion2}:
\begin{equation}\label{volterra}
u(x)\in g\f(x,u(x),\int_{\Lambda(x)}k(x,y)F(y,u(y))\,dy\g),\;\;x\in\Omega,
\end{equation}
where $g\colon\Omega\times E\times E\to E$ satisfies
\begin{itemize}
\item[$(\gie_1')$] $g$ is uniformly continuous on bounded subsets of $\Omega\times E\times E$,
\item[$(\gie_2')$] there exists a nondecreasing positively homogeneous usc at zero map $\vartheta\colon\R{}_+\to\R{}_+$ such that $\vartheta(x)\<x$ for $x\in\R{}_+$ and a concave function $\varphi\in\phi$ satisfying \eqref{10} for which \[|g(x,u_1,u_2)-g(x,w_1,w_2)|\<\varphi(|u_1-w_1|)+\vartheta(|u_2-w_2|)\] on $\Omega\times E\times E$.
\end{itemize}

\begin{theorem}\label{g}
Assume ${\bm\Omega}(\Lambda)\neq\varnothing$ and ${\bm\tau}(\Lambda)\neq\varnothing$. Let $(\mathbb{E})$ be satisfied. Assume that conditions $(\ka_1)$-$(\ka_2)$, $(\gie_1')$--$(\gie_2')$ and $(\F_1)$--$(\F_5)$ hold. If the following inequality is satisfied 
\begin{equation}\label{brzeg2}
\liminf_{n\to\infty}\,\f(a_n-\varphi(a_n)-||g(\cdot,0,0)||_n\g)>0
\end{equation}
for some $\z{a_n}\in\R{\mathbb{N}}_+$, then the solution set of Volterra integral inclusion \eqref{volterra} is nonempty.% and compact in the compact-open topology of the space $C(\Omega,E)$.
\end{theorem}

\begin{remark}
Each concave function $\vartheta\in\phi$ meets demands of the proof of Theorem \ref{g}.
\end{remark}

\begin{proof}
Fix $\z{\Omega_n}\in{\bm\Omega}(\Lambda)$ and $\tau\in{\bm\tau}(\Lambda)$. Define multimaps ${\mathcal F},H\colon C(\Omega,E)\map C(\Omega,E)$ in the following way ${\mathcal F}:=V\circ N_F$ and $H:=N_g\circ(I\times{\mathcal F})$. As shown previously the operator $I\times{\mathcal F}\colon C(\Omega,E)\map C(\Omega,E)\times C(\Omega,E)$ is usc with compact convex values. Since $N_g\colon C(\Omega,E)\times C(\Omega,E)\to C(\Omega,E)$ is continuous, the multimap $H$ is admissible.
\par Taking into account assumption \eqref{brzeg2} and upper semicontinuity of $\vartheta$ at zero, we may choose $(L_n)_{n=N}^\infty\subset\R{}_+$ such that \[\vartheta\f(\sup_{x\in\Omega_n}||K(x)||_{L^\infty_n}\Phi(L_n,b)_n(1+a_n)\g)\<a_n-\varphi(a_n)-||g(\cdot,0,0)||_n.\] For $v\in H(u)\subset H({\mathcal X})$ and $n\geqslant N$ one has
\begin{align*}
||v||_{L_n}&\<\varphi(||u||_{L_n})+||g(\cdot,0,0)||_n+\vartheta\f(\sup_{x\in\Omega_n}||K(x)||_{L^\infty_n}\Phi(L_n,b)_n\f(1+||u||_{L_n}\g)\g)\\&\<\varphi(a_n)+||g(\cdot,0,0)||_n+\vartheta\f(\sup_{x\in\Omega_n}||K(x)||_{L^\infty_n}\Phi(L_n,b)_n(1+a_n)\g)\<a_n
\end{align*}
Therefore, $H({\mathcal X})\subset{\mathcal X}$.\par Let $\hat{L}\in\R{\mathbb{N}\setminus\{1,\ldots,N-1\}}_+$ be such that \[4\sup\limits_{x\in\Omega_n}||K(x)||_{L^\infty_n}\Phi(\hat{L}_n,\eta)_n\<k_n,\] where $k_n\in(0,1)$ is the constant introduced in \eqref{r_n}. Then
\begin{equation}\label{wacek}
\begin{split}
\varphi(x)+\vartheta\f(4\sup\limits_{x\in\Omega_n}||K(x)||_{L^\infty_n}\Phi(\hat{L}_n,\eta)_nx\g)&\<\varphi(x)+4\sup\limits_{x\in\Omega_n}||K(x)||_{L^\infty_n}\Phi(\hat{L}_n,\eta)_nx\\&\<\varphi(x)+k_nx
\end{split}
\end{equation}
for every $n\geqslant N$ and each $x\in(0,r_n]$. Suppose that $M\subset{\mathcal X}$ is not relatively compact. Observe that
\begin{align*}
\beta(H(M)(x))&=\beta(\{g(x,u(x),V(w)(x))\colon u\in M, w\in N_F(u)\})\<\beta(g(\{x\}\times M(x)\times{\mathcal F}(M)(x)))\\&\<\varphi(\beta(M(x)))+\psi(\beta({\mathcal F}(M)(x))).
\end{align*}
for every $x\in\Omega$. Therefore, taking into account \eqref{szacowanie} and \eqref{wacek}, we arrive at \[\beta_{\hat{L}_n}(H(M))\<\varphi\f(\beta_{\hat{L}_n}(M)\g)+\vartheta\f(4\sup_{x\in\Omega_n}||K(x)||_{L^\infty_n}\Phi(\hat{L}_n,\eta)_n\,\beta_{\hat{L}_n}(M)\g)\<\psi_n\f(\beta_{\hat{L}_n}(M)\g)\] for $n\geqslant N$. Since $g$ is in particular uniformly continuous on the set \[\Omega_{n+1}\times D(0,r_{n+1})\times D\f(0,\sup\limits_{x\in\Omega_{n+1}}||K(x)||_{L^\infty_{n+1}}||b||_{L^1\big(\cl_\Omega(\Omega_{n+1})\big)}\f(1+r_{n+1}\g)\g),\] we see that
\begin{equation}\label{uniform}
\sup_{x\in\Omega_n}\overline{\lim_{z\to x}}\sup_{\stackrel{\scriptstyle u\in M}{w\in N_F(u)}}|g(x,u(z),V(w)(z))-g(z,u(z),V(w)(z))|=0.
\end{equation}
It follows from \eqref{V} that $\overline{\lim\limits_{z\to x}}\sup\limits_{w\in N_F(M)}|V(w)(x)-V(w)(z)|=0$ for every $x\in\Omega_n$. Whence, for all $x\in\Omega_n$ 
\begin{equation}\label{fiftaszek}
\overline{\lim_{z\to x}}\,\vartheta\f(\sup_{w\in N_F(M)}|V(w)(x)-V(w)(z)|\g)\<\overline{\lim_{z\to 0^+}}\vartheta(z)\<0,
\end{equation}
because $\vartheta$ is usc at zero. In accordance with by \eqref{uniform} and \eqref{fiftaszek}, one may estimate
\begin{align*}
e_n(H(M))&=\sup_{x\in\Omega_n}\overline{\lim_{z\to x}}\sup_{\stackrel{\scriptstyle u\in M}{w\in N_F(u)}}|g(x,u(x),V(w)(x))-g(z,u(z),V(w)(z))|\\&\<\sup_{x\in\Omega_n}\overline{\lim_{z\to x}}\sup_{\stackrel{\scriptstyle u\in M}{w\in N_F(u)}}|g(x,u(x),V(w)(x))-g(x,u(z),V(w)(z))|\\&\<\sup_{x\in\Omega_n}\overline{\lim_{z\to x}}\sup_{u\in M}\varphi(|u(x)-u(z)|)+\sup_{x\in\Omega_n}\overline{\lim_{z\to x}}\sup_{{w\in N_F(M)}}\vartheta(|V(w)(x)-V(w)(z)|)\\&\<\varphi(e_n(M)).
\end{align*}
It becomes clear, therefore, that the previously obtained estimation remains in force i.e., \[\frac{1}{2}\nu^N_{\hat{L}}(H(M))_n\<\psi_n\f(\frac{1}{2}\beta_{\hat{L}_n}(M)+\frac{1}{2}e_n(M)\g).\] Completely analogous reasoning as in the proof of Theorem \ref{ex1} leads to the conclusion that the multimap $H\colon{\mathcal X}\map{\mathcal X}$ meets the assumptions of Theorem \ref{fixed}. The latter means that the solution set of the integral inclusion \eqref{volterra} is nonempty.
\end{proof}

\begin{corollary}
Assume ${\bm\Omega}(\Lambda)\neq\varnothing$ and ${\bm\tau}(\Lambda)\neq\varnothing$. Let $(\mathbb{E})$ be satisfied. Suppose that hypotheses $(\ka_1)$-$(\ka_2)$  and $(\F_1)$-$(\F_5)$ hold. If assumptions $(\gie_1')$-$(\gie_2')$ are satisfied with the proviso that $\varphi\in\phi$ is given by $\varphi(x):=kx$ for some $k\in(0,1)$ and $R:=\sup\limits_{x\in\Omega}|g(x,0,0)|<\infty$, then the solution set of the Volterra integral inclusion \eqref{volterra} is nonempty and compact in the compact-open topology of $C(\Omega,E)$.
\end{corollary}
%\begin{equation}\label{volterra2}
%u(x)\in G\f(x,u(x),\int_{\Lambda(x)}k(x,y)F(y,u(y))\,dy\g),\;\;x\in\Omega,
%\end{equation}
The third problem to which we give a careful consideration is the integral inclusion of the form \eqref{general} with the proviso that $G\colon\Omega\times\R{M}\times\R{M}\map\R{M}$ satisfies
\begin{itemize}
\item[$(\text{G}_1)$] $G$ has compact convex values,
\item[$(\text{G}_2)$] for every $(x_1,u_1,w_1),(x_x,u_2,w_2)\in\Omega\times E\times E$ one has \[h(G(x_1,u_1,w_1),G(x_2,u_2,w_2))\<L\max\{|x_1-x_2|,|u_1-u_2|,|w_1-w_2|\}\] with \[L<\frac{\sqrt{\pi}\,\Gamma\f(\frac{M+1}{2}\g)}{2\Gamma\f(\frac{M}{2}+1\g)}\]
\end{itemize}
and $F\colon\Omega\times\R{M}\map\R{M}$ is the set-valued map such that
\begin{itemize}
\item[$(\F_1^M)$] for every $(x,u)\in \Omega\times\R{M}$ the set $F(x,u)$ is nonempty compact and convex,
\item[$(\F_2^M)$] the map $F(\cdot,u)$ has a measurable selection for every $u\in\R{M}$,
\item[$(\F_3^M)$] the map $F(x,\cdot)$ is upper semicontinuous for a.a. $x\in\Omega$,
\item[$(\F_4^M)$] there exists $b\in L^1_{\text{loc}}(\Omega)$ such that \[||F(x,u)||^+\<b(x)(1+|u|)\;\text{ a.e. on }\Omega, \text{ for all }u\in\R{M}.\]
\end{itemize}

\begin{theorem}
Assume ${\bm\Omega}(\Lambda)\neq\varnothing$ and ${\bm\tau}(\Lambda)\neq\varnothing$. Suppose that hypotheses $(\ka_1)$-$(\ka_2)$, $(\operatorname{G}_1)$--$(\operatorname{G}_2)$ and $(\F_1^M)$--$(\F_4^M)$  hold. Then \eqref{general} has at least one continuous solution.
\end{theorem}

\begin{proof}
Fix $\z{\Omega_n}\in{\bm\Omega}(\Lambda)$ and $\tau\in{\bm\tau}(\Lambda)$. Let ${\mathcal H}(\R{M})$ denote the space of nonempty convex compact subsets of $\R{M}$, endowed with the Hausdorff-Pompeiu metric. In view of \cite[Proposition 2.19]{linden} the Steiner point map $S\colon{\mathcal H}(\R{M})\to\R{M}$ is a Lipschitz selection with Lipschitz constant $2\pi^{-\frac{1}{2}}\Gamma\f(\frac{M}{2}+1\g)/\Gamma\f(\frac{M+1}{2}\g)$. Define $g\colon\Omega\times\R{M}\times\R{M}\to\R{M}$ by $g:=S\circ G$. Then $g$ is a Lipschitz selection of $G$ with Lipschitz constant \[\w{L}:=\frac{2L\Gamma\f(\frac{M}{2}+1\g)}{\sqrt{\pi}\Gamma\f(\frac{M+1}{2}\g)}<1.\]\par If the domain $\Omega$ is unbounded, then $\z{||\Omega_n||^+}$ converges to infinity. Put $a_n:=k||\Omega_n||^+$ with $k>(1-\w{L})^{-1}$. This definition enables us to estimate
\begin{align*}
\liminf_{n\to\infty}\f(a_n-\varphi(a_n)-||g(\cdot,0,0)||_n\g)&\geqslant\liminf_{n\to\infty}\f(a_n-\w{L}a_n-\w{L}||\Omega_n||^+-||G(0,0,0)||^+\g)\\&=\lim_{n\to\infty}\f((1-\w{L})k-1\g)||\Omega_n||^+-||G(0,0,0)||^+=+\infty.
\end{align*}
Suppose, then, that $\Omega$ is bounded. Since $\sup\limits_\n||\Omega_n||^+<+\infty$, one has
\begin{align*}
\liminf_{n\to\infty}\f(a_n-\varphi(a_n)-||g(\cdot,0,0)||_n\g)&\geqslant\liminf_{n\to\infty}\f((1-\w{L})a_n-\w{L}\sup_\K||\Omega_k||^+\!-||G(0,0,0)||^+\g)\!=\!+\infty,
\end{align*}
for any $\z{a_n}\in\R{\mathbb{N}}_+$ with $\lim\limits_{n\to\infty}a_n=+\infty$. These arguments justify \eqref{brzeg2}.\par It is clear that $g$ satisfies $(\gie_1')$--$(\gie_2')$. Since $(\F_1)$--$(\F_5)$ also hold, the integral inclusion \eqref{volterra} possesses a solution, by Theorem \ref{g}. Obviously, this is also a solution of \eqref{general}.
\end{proof}

The observation that the uniform continuity of the selection $g$ of the map $G\colon\Omega\times E\map E$ is sufficient from the point of view of the solutions' existence is confirmed in the following theorem:
\begin{theorem}
Assume ${\bm\Omega}(\Lambda)\neq\varnothing$ and ${\bm\tau}(\Lambda)\neq\varnothing$. Let $E$ be a uniformly convex Banach space. Suppose that hypotheses $(\ka_1)$-$(\ka_2)$, $(\F_1)$-$(\F_5)$ and 
\begin{itemize}
\item[$(\text{G}'_1)$] $G$ is a multivalued map with nonempty convex compact values,
\item[$(\text{G}'_2)$] there exist upper semicontinuous functions $\theta,\varphi\colon\R{}_+\to\R{}_+$ such that $\theta(0)=0$ and $\varphi(x)\<x$ for $x\in\R{}_+$, for which one has \[h(G(x,u),G(y,w))\<\theta(|x-y|)+\varphi(|u-w|)\] on $\Omega\times E$.
\end{itemize}
hold. Further, assume that 
\begin{equation}\label{popierdolka}
\liminf_{n\to\infty}\f(a_n-\sup_{x\in\Omega_n}\theta(|x|)\g)>||G(0,0)||^+
\end{equation}
for some $\z{a_n}\in\R{\mathbb{N}}_+$. Then the following integral inclusion
\begin{equation}\label{hj}
u(x)\in G\f(x,\int_{\Lambda(x)}k(x,y)F(y,u(y))\,dy\g),\;\;x\in\Omega
\end{equation}
has at least one continuous solution.
\end{theorem}

\begin{remark}
{\em Assumption \eqref{popierdolka} is achievable. Indeed, suppose for instance that \[\limsup\limits_{x\to+\infty}\frac{\theta(x)}{x}<1.\] Since $\theta$ is usc, $\sup\limits_{x\in\Omega_n}\theta(|x|)\<\theta(|x_n|)$ for some $x_n\in\overline{\Omega}_n$. Then we are dealing with two possible cases. Let us first assume that $\sup\limits_{n\in\mathbb{N}}|x_n|<+\infty$. Then $\sup\limits_{n\in\mathbb{N}}\theta(|x_n|)<\infty$ and \[\liminf_{n\to\infty}(a_n-\sup\limits_{x\in\Omega_n}\theta(|x|))=+\infty\] for each $\z{a_n}\in\R{\mathbb{N}}_+$ with $\lim\limits_{n\to\infty}a_n=+\infty$. If there is the case $\lim\limits_{n\to\infty}|x_n|=+\infty$, then 
\begin{align*}
\liminf_{n\to\infty}(a_n-\sup\limits_{x\in\Omega_n}\theta(|x|))&:=\liminf_{n\to\infty}((L+1)||\Omega_n||^+-\sup\limits_{x\in\Omega_n}\theta(|x|))\\&\geqslant\liminf_{n\to\infty}((L+1)||\Omega_n||^+-\theta(|x_n|))\geqslant\liminf_{n\to\infty}((L+1)||\Omega_n||^+-L|x_n|)\\&\geqslant\liminf_{n\to\infty}((L+1)||\Omega_n||^+-L||\Omega_n||^+)=\lim_{n\to\infty}||\Omega_n||^+>0,
\end{align*} 
where $\limsup\limits_{x\to+\infty}\frac{\theta(x)}{x}<L<1$.}
\end{remark}

\begin{proof}
Let ${\mathcal H}(E)$ denote the space of nonempty closed convex and bounded subsets of $E$, endowed with the Hausdorff-Pompeiu metric. By virtue of \cite[Theorem 1.24]{linden} there exists a selector $\phi\colon{\mathcal H}(E)\to E$ which is uniformly continuous on bounded subsets of ${\mathcal H}(E)$. Define $g\colon\Omega\times E\to E$ by $g(x,u):=\phi(G(x,u))$. Observe that \[||G(x,u)||^+\<\theta(|x|)+\varphi(|u|)+||G(0,0)||^+.\] This means that $G$ maps bounded subsets of $\Omega\times E$ into bounded subsets of $E$. Assumption $(\text{G}_2)$ and the upper semicontinuity of $\theta$ and $\varphi$ at zero imply uniform continuity of $g$ on bounded subsets.\par Define $H\colon C(\Omega,E)\map C(\Omega,E)$ by the formulae $H:=N_g\circ{\mathcal F}$. It is a matter of routine to check that $N_g\in C(C(\Omega,E),C(\Omega,E))$. As we have managed to appoint previously, the map ${\mathcal F}$ is admissible. Thus, $H$ must be admissible. \par In view of \eqref{popierdolka} one has \[\liminf_{n\to\infty}\f(a_n-\sup_{x\in\Omega_n}\theta(|x|))-||G(0,0)||^+\g)>0,\] which means that one may choose $(L_n)_{n=N}^\infty\subset\R{}_+$ in such a way that \[\Phi(L_n,b)_n\<\frac{a_n-\sup\limits_{x\in\Omega_n}\theta(|x|))-||G(0,0)||^+}{\sup\limits_{x\in\Omega_n}||K(x)||_{L^\infty_n}(1+a_n)}.\] Let ${\mathcal X}$ be given by \eqref{X}. For $v\in H({\mathcal X})$ and $n\geqslant N$ one has \[||v||_{L_n}\<\sup_{x\in\Omega_n}||K(x)||_{L^\infty_n}\Phi(L_n,b)_n(1+a_n)+\sup_{x\in\Omega_n}\theta(|x|)+||G(0,0)||^+\<a_n.\] Hence, $H({\mathcal X})\subset{\mathcal X}$.\par Suppose that $M\subset{\mathcal X}$ is not relatively compact. Observe that \[\beta(H(M)(x))\<\beta(g(\{x\}\times{\mathcal F}(M)(x)))\<\beta(G(\{x\}\times{\mathcal F}(M)(x)))\<\varphi(\beta({\mathcal F}(M)(x)))\] for every $x\in\Omega$ (the assumption that $G$ is compact valued is here indispensable). In view of Lemma \ref{lemcon} one may choose sequences $\hat{L}\in\R{\mathbb{N}\setminus\{1,\ldots,N-1\}}_+$ and  $(k_n)_{n=N}^\infty$ in the following way 
\begin{equation}\label{tutaj3}
4\sup_{x\in\Omega_n}||K(x)||_{L^\infty_n}\Phi(\hat{L}_n,\eta)_n<k_n<1.
\end{equation}
Therefore, in view of \eqref{szacowanie}
\begin{equation}\label{tutaj2}
\begin{split}
\beta_{\hat{L}_n}(H(M))&\<\sup_{x\in\Omega_n}e^{-\hat{L}_n\tau(x)}\varphi\f(\beta({\mathcal F}(M)(x))\g)\<\sup_{x\in\Omega_n}e^{-\hat{L}_n\tau(x)}\beta({\mathcal F}(M)(x))=\beta_{\hat{L}_n}({\mathcal F}(M))\\&\<4\sup_{x\in\Omega_n}||K(x)||_{L^\infty_n}\Phi(\hat{L}_n,\eta)_n\,\beta_{\hat{L}_n}(M)
\end{split}
\end{equation}
for $n\geqslant N$. Since $g$ is in particular uniformly continuous on the set $\Omega_{n+1}\times D(0,R)$ with $R:=\sup\limits_{x\in\Omega_{n+1}}||K(x)||_{L^\infty_{n+1}}||b||_{L^1\big(\cl_\Omega(\Omega_{n+1})\big)}\f(1+r_{n+1}\g)$, we see that \[\sup_{x\in\Omega_n}\overline{\lim_{z\to x}}\sup_{\stackrel{\scriptstyle u\in M}{w\in N_F(u)}}|g(x,V(w)(z))-g(z,V(w)(z))|=0.\]
Moreover, since $\Omega_n$ is precompact and $e_n({\mathcal F}(M))=0$ one easily sees that for every $\eps>0$
\[\exists\,\delta>0\,\forall\,x\in\Omega_n\,\forall\,u_1,u_2\in D(0,R)\;\;\;|u_1-u_2|<\delta\Rightarrow|g(x,u_1)-g(x,u_2)|<\eps\] and \[\exists\,\gamma>0\,\forall\,x\in\Omega_n\,\forall\,z\in B(x,\gamma)\;\;\;\sup_{w\in N_F(M)}|V(w)(x)-V(w)(z)|<\delta.\] In other words, for every $\eps>0$ \[\sup_{x\in\Omega_n}\inf_{\gamma>0}\sup_{z\in B(x,\gamma)}\sup_{w\in N_F(M)}|g(x,V(w)(x))-g(x,V(w)(z))|<\eps.\]
It follows that for each $n\geqslant N$
\begin{equation}\label{tutaj}
\begin{split}
e_n(H(M))&=\sup_{x\in\Omega_n}\overline{\lim_{z\to x}}\sup_{w\in N_F(M)}|g(x,V(w)(x))-g(z,V(w)(z))|\\&\<\sup_{x\in\Omega_n}\overline{\lim_{z\to x}}\sup_{w\in N_F(M)}|g(x,V(w)(x))-g(x,V(w)(z))|\\&=0.
\end{split}
\end{equation}
Considering properties \eqref{tutaj2} and \eqref{tutaj} one sees that \[\nu^N_{\hat{L}}(H(M))_n\<4\sup_{x\in\Omega_n}||K(x)||_{L^\infty_n}\Phi(\hat{L}_n,\eta)_n\,\beta_{\hat{L}_n}(M)\<4\sup_{x\in\Omega_n}||K(x)||_{L^\infty_n}\Phi(\hat{L}_n,\eta)_n\,\nu^N_{\hat{L}}(M)_n\] for $n\geqslant N$. Taking into account coefficients $\z{k_n}$ characterized by \eqref{tutaj3} we may define the function $f\colon\R{\mathbb{N}\setminus\{1,\ldots,N-1\}}_+\times\R{\mathbb{N}\setminus\{1,\ldots,N-1\}}_+\to\R{\mathbb{N}\setminus\{1,\ldots,N-1\}}$ using formulae \eqref{f}. It is clear that \[f\big(\nu^N_{\hat{L}}(H(M)),\nu^N_{\hat{L}}(M)\big)\in\R{\mathbb{N}\setminus\{1,\ldots,N-1\}}_+\setminus\{0\}.\] Hence the assumption \eqref{cond} of Theorem \ref{fixed} is met and the existence of fixed points of $H$ follows.
\end{proof}

\section{Examples}

\begin{example}
{\em Let's modify \cite[Example 4.1]{arab} a bit. Consider the following equation
\begin{equation}\label{Example1}
x(t)=te^{-(1+t^2)}+\ln(\lambdaup+|x(t)|)+\int\limits_{\sin t}^{|t|}e^{t^2}(\cos(x(s))+2)\,ds,\;\;t\in\R{},
\end{equation}
where $\lambdaup>1$. It is easy to see that \[\sup\f\{\f|\,\int\limits_{\sin t}^{|t|}e^{t^2}(\cos(x(s))+2)\,ds\g|\colon t\in\R{}, x\in BC(\R{})\g\}=+\infty\] and \[\lim_{|t|\to\infty}\f|\,\int\limits_{\sin t}^{|t|}e^{t^2}(\cos(0)+2)-e^{t^2}\f(\cos\f(\frac{\pi}{2}\g)+2\g)\,ds\g|=+\infty.\] Therefore, the application of \cite[Theorem 3.1]{arab} must fail. However, assumptions $(\ka_1)$-$(\ka_2)$, $(\gie_1)$-$(\gie_2)$ and $(\F_1)$-$(\F_5)$ are satisfied for 
\[\begin{cases}
k(t,s):=\exp(t^2)\\ 
F(t,x):=\cos(x)+2\\ 
g(t,x):=te^{-(1+t^2)}+\ln(\lambdaup+|x|).
\end{cases}\]
Define $\Lambda\colon\R{}\to{\mathfrak L}(\R{})$ by $\Lambda(t):=(\sin t,|t|)$ and $\Omega_n:=(-n,n)$. Clearly, equation \eqref{Example1} poses a particular case of the inclusion \eqref{inclusion2}. Since $||\Lambda(t)||^+\<|t|$ for $t\in\R{}$, one has $\z{\Omega_n}\in{\bm\Omega}(\Lambda)$ and ${\bm\tau}(\Lambda)\neq\varnothing$. Furthermore, condition \eqref{brzeg} is met, because
\begin{align*}
\liminf_{n\to\infty}(n-\varphi(n)-||g(\cdot,0)||_n)&=\lim_{n\to\infty}\f(a_n-\frac{1}{\lambdaup}a_n-\sup_{t\in(-n,n)}|t|e^{-(1+t^2)}-\ln\lambdaup\g)\\&\geqslant\lim_{n\to\infty}\f(\f(1-\frac{1}{\lambdaup}\g)a_n-e^{-1}n-\ln\lambdaup\g)=+\infty
\end{align*}
for $a_n:=kn$ with $k>\frac{\lambdaup e^{-1}}{\lambdaup-1}$.
In connection with the above, equation \eqref{Example1} has at least one continuous solution by virtue of Theorem \ref{ex1}.}
\end{example}

\begin{example}
{\em Theorem 3.1 in \cite{arab} is failing even in the case of the most elementary Volterra equations of the second kind as the following example illustrates:
\begin{equation}\label{Example2}
u(x)=A+\int\limits_a^xu(y)\,dy,\;\;\;x\in(a,\infty)
\end{equation}
with $a\geqslant 0$ and $A\in\R{}$. Obviously, equation \eqref{Example2} possesses a unique continuous solution $u_0\colon(a,\infty)\to\R{}$ of the form $u_0(x);=A\exp(x-a)$. This function is unbounded, so \cite[Theorem 3.1]{arab} does not detect it.\par Define $\Lambda\colon(a,\infty)\to{\mathfrak L}((a,\infty))$ by $\Lambda(x):=(a,x)$ and $\Omega_n:=(a,a+n)$. Then $\z{\Omega_n}\in{\bm\Omega}(\Lambda)$. Moreover, ${\bm\tau}(\Lambda)\neq\varnothing$. Let
\[\begin{cases}
k(x,y):=1\\ 
F(x,u):=\{u\}\\ 
g(x,u):=A,\\
\varphi(x):=kx\text{ for some }k\in(0,1).
\end{cases}\]
Clearly, assumptions $(\ka_1)$-$(\ka_2)$, $(\gie_1)$-$(\gie_2)$ and $(\F_1)$-$(\F_5)$ are met. At the same time \[\liminf_{n\to\infty}(n-\varphi(n)-||g(\cdot,0)||_n)=\lim_{n\to\infty}((1-k)n-A)=+\infty.\] It is therefore clear that Theorem \ref{ex1} does detect the existence of the solution $u_0$.}
\end{example}

\begin{example}
{\em Consider the following problem:
\begin{equation}\label{wave}
\begin{dcases*}
u_{tt}-\Delta u=g_1\star f_2(t)+\Delta\int_0^tg_2\star f_1(s)\,ds&in $(0,\infty)\times\R{N}$\\
f_1(t,x)\in\f[h^1_1\f(t,x,\int_{\R{n}}k_1(t,y)u(t,y)\,dy\g),h^1_2\f(t,x,\int_{\R{n}}k_2(t,y)u(t,y)\,dy\g)\g]&in $(0,\infty)\times\R{N}$\\
f_2(t,x)\in\f[h^2_1\f(t,x,\int_{\R{n}}k_1(t,y)u(t,y)\,dy\g),h^2_2\f(t,x,\int_{\R{n}}k_2(t,y)u(t,y)\,dy\g)\g]&in $(0,\infty)\times\R{N}$\\
u_t(0)=\mathring{u}_2&on $\R{n}$\\
u(0)=\mathring{u}_1&on $\R{n}$,
\end{dcases*}
\end{equation}
where $\Delta$ is the Laplace operator, $g_i\in L^1(\R{N})$ and $k_i(t,\cdot)\in L^2(\R{N})$ for a.a. $t\in(0,\infty)$ and $i=1,2$. Let $\langle\cdot,\cdot\rangle$ denote the inner product in $L^2(\R{N})$. 
\begin{definition}
By the weak solution of the problem \eqref{wave} we mean $w\in C(\R{}_+,L^2(\R{N}))$ such that for every $v\in H^2(\R{N})$ the function $\langle w(\cdot),v\rangle$ is twice differentiable and $w$ satisfies
\begin{equation*}
\begin{dcases*}
\frac{d^2}{dt^2}\,\langle w(t),v\rangle=\langle w(t),\Delta v\rangle+\langle g_2\star f_2(t),v\rangle+\f\langle\int_0^tg_1\star f_1(s)\,ds,\Delta v\g\rangle&a.e. on $(0,\infty)$\\
\res{\frac{d}{dt}\,\langle w(t),v\rangle}{t=0}=\langle\mathring{u}_2,v\rangle\\
w(0)=\mathring{u}_1
\end{dcases*}
\end{equation*}
for some functions $f_1,f_2\in L_{\text{loc}}^1(\R{}_+,L^2(\R{N}))$ such that
\begin{equation*}
\begin{dcases}
h_1^1\f(t,x,\int_{\R{N}}\int_0^tk_1(t,y)w(s,y)\,dsdy\g)\<f_1(t,x)\<h_2^1\f(t,x,\int_{\R{N}}\int_0^tk_2(t,y)w(s,y)\,dsdy\g)\\
h_1^2\f(t,x,\int_{\R{N}}\!k_1(t,y)(w(t,y)-\mathring{u}_1(y))\,dy\g)\<f_2(t,x)\<h_2^2\f(t,x,\int_{\R{N}}\!k_2(t,y)(w(t,y)-\mathring{u}_1(y))\,dy\g)
\end{dcases}
\end{equation*}
for a.a. $t\in(0,\infty)$ and a.a. $x\in\R{N}$.
\end{definition}
\par Our hypotheses on $h_j^i\colon(0,\infty)\times\R{N}\times\R{}\to\R{}$ are the following:
\begin{itemize}
\item[$(\text{h}_1)$] for $i=1,2$ and for any $u\in L^2(\R{N})$ there exists $v\in L^1_{\text{loc}}(\R{}_+,L^2(\R{N}))$ such that \[h^i_1\f(t,x,\int_{\R{N}}k_1(t,y)u(y)\,dz\g)\<v(t,x)\<h^i_2\f(t,x,\int_{\R{N}}k_2(t,y)u(y)\,dy\g)\] for a.a. $t\in(0,\infty)$ and a.a. $x\in\R{N}$,
\item[$(\text{h}_2)$] for $i=1,2$, for a.a. $t\in(0,\infty)$ and for a.a. $x\in\R{N}$ the functions $h^i_1(t,x,\cdot)$ are lower semicontinuous while $h^i_2(t,x,\cdot)$ are upper semicontinuous,
\item[$(\text{h}_3)$] for $i,j=1,2$ there exists $b_i\in L^1_{\text{loc}}(\R{}_+)$ and $c_i\colon(0,\infty)\times\R{N}\times\R{}_+\to\R{}$ such that \[\sup_{|z|\<||k_j(t,\cdot)||_2r}|h_j^i(t,x,z)|\<c_i(t,x,r)\] and \[\int_{\R{N}}c_i^2(t,x,r)\,dx\<b_i^2(t)(1+r)^2\] for every $r>0$, for a.a. $t\in I$ and for a.a. $x\in\R{N}$.
\end{itemize}}

\begin{theorem}
If hypotheses $(\text{h}_1)$-$(\text{h}_3)$ hold, then for every $\mathring{u}_1,\mathring{u}_2\in L^2(\R{N})$ problem \eqref{wave} possesses a weak solution. 
\end{theorem}
\begin{proof}
Let $\Omega:=(0,\infty)$, $E:=L^2(\R{n})\times L^2(\R{n})$ and $D(A):=H^2(\R{n})\times L^2(\R{n})$. Assume that the Hilbert space $E$ is furnished with the norm \[||(x,y)||_E:=\f(||x||_2^2+||y||_2^2\g)^\frac{1}{2}.\] The linear operator $A\colon D(A)\to E$, given by $A(u_1,u_2):=(u_2,\Delta u_1)$, generates an exponentially bounded non-degenerate integrated semigroup $\{S(t)\}_{t\geqslant 0}$ on $E$ such that \[S(t)(\mathring{u}_1,\mathring{u}_2)=\f(\int_0^tw(s)\,ds,w(t)-\mathring{u}_1\g),\] where $w\in C^2([0,\infty),L^2(\R{n}))$ satisfies
\begin{equation*}
\begin{dcases}
\frac{d^2}{dt^2}\,\langle w(t),v\rangle=\langle w(t),\Delta v\rangle\\
\res{\frac{d}{dt}\,\langle w(t),v\rangle}{t=0}=\langle\mathring{u}_2,v\rangle\\
w(0)=\mathring{u}_1
\end{dcases}
\end{equation*}
for every $v\in H^2(\R{n})$ (see \cite[Th.7.1.]{thieme}).\par For $i=1,2$ define $F_i\colon\Omega\times L^2(\R{n})\map L^2(\R{n})$ by the formula
\begin{align*}
&F_i(t,u):=\\&\f\{v\in L^2(\R{n})\colon h^i_1\f(t,x,\int_{\R{n}}\!k_1(t,y)u(y)\,dy\g)\<v(x)\<h^i_2\f(t,x,\int_{\R{n}}\!k_2(t,y)u(y)\,dy\g)\text{ a.e. on }\R{n}\g\}\!.
\end{align*}
Let $F\colon\Omega\times E\map E$ be a map given by $F(t,u_1,u_2):=(g_1\star F_1(t,u_1))\times(g_2\star F_2(t,u_2))$. Consider the following Volterra integral inclusion
\begin{equation}\label{waveeq}
u(t)\in S(t)(\mathring{u}_1,\mathring{u}_2)+\int_0^tS(t-s)F(s,u(s))\,ds,\;\;\;t\in\Omega.
\end{equation}
Clearly, the above inclusion poses a special case of the problem \eqref{inclusion2}.\par Fix $i\in\{1,2\}$ and $u\in L^2(\R{N})$. Let $v_i\in L^1_{\text{loc}}(\R{}_+,L^2(\R{N}))$ be the mapping existing in view of the assumption $(\text{h}_1)$. Let $\z{v_n^i\colon\Omega\to L^2(\R{N})}$ be a sequence of simple functions such that $v_n^i(t)\xrightarrow[n\to\infty]{L^2(\R{N})}v_i(t)$ a.e. on $\Omega$. By Young's inequality \[||g_i\star v_n^i(t)-g_i\star v_i(t)||_2=||g_i\star(v_n^i-v_i)(t)||_2\<||g_i||_1||v_n^i(t)-v_i(t)||_2.\] Whence $g_i\star v_n^i(t)\xrightarrow[n\to\infty]{L^2(\R{N})}g_i\star v_i(t)$ a.e. on $\Omega$ i.e., the function $g_i\star v_i(\cdot)$ is measurable. Eventually, $(g_1\star v_1(\cdot))\times(g_2\star v_2(\cdot))$ poses a strongly measurable selection of the multimap $F(\cdot,u)$. In this manner assumption  $(\F_2)$ has been verified.\par Take $(u_1,u_2)\in E$ and $(g_1\star f_1,g_2\star f_2)\in F(t,(u_1,u_2))$. Then \[|f_i(x)|\<\max\f\{\f|h_1^i\f(t,x,\int_{\R{n}}k_1(t,y)u_i(y)\,dy\g)\g|,\f|h_2^i\f(t,x,\int_{\R{n}}k_2(t,y)u_i(y)\,dy\g)\g|\g\}\<c_i(t,x,||u_i||_2)\] and $||f_i||_2\<b_i(t)(1+||u_i||_2)$. Whence 
\begin{align*}
||F(t,(u_1,u_2))||^+_2&\<||g_1||_1b_1(t)(1+||u_1||_2)+||g_2||_2b_2(t)(1+||u_2||_2)\\&\<\max\{||g_1||_1,||g_2||_1\}(b_1(t)+b_2(t))(1+||(u_1,u_2)||_E)
\end{align*}
i.e., $(\F_4)$ is met. Let $M\subset L^2(\R{N})$ be bounded. Since $F_i(\{t\}\times M)$ is relatively weakly comapct, for each $\eps>0$ there is a measurable and bounded subset $\Omega_\eps\subset\R{N}$ such that \[\sup_{f\in F_i(\{t\}\times M)}||g_i\star f||_{L^2(\R{N}\setminus\Omega_\eps)}\<||g_i||_1\sup_{f\in F_i(\{t\}\times M)}||f||_{L^2(\R{N}\setminus\Omega_\eps)}<\eps,\] in view of the Dunford-Pettis theorem. On the other hand the set $g_i\star F_i(\{t\}\times M)$ is $2$-equiintegrable (cf. \cite[Corollary 4.28]{brezis}). Therefore, $g_i\star F_i(\{t\}\times M)$ satisfies hypotheses of the Riesz-Kolmogorov theorem. Eventually, the image $F(\{t\}\times M)$ is relatively compact in the norm topology of $E$.\par Since the operator $g_i\star(\cdot)\colon L^2(\R{N})\to L^2(\R{N})$ is linear and continuous in the norm topology, weak convergence $f_n\xrightharpoonup[n\to\infty]{L^2(\R{N})}f$ entails $g_i\star f_n\xrightharpoonup[n\to\infty]{L^2(\R{N})}g_i\star f$. Therefore, the reiteration of the arguments contained in the proof of \cite[Theorem 8]{pietkun} leads to the conclusion that the graph $\gr(F(t,\cdot))$ is sequentially closed in $(E,||\cdot||_E)\times(E,w)$ for a.a. $t\in\Omega$. Considering that the set-valued map $F(t,\cdot)\colon(E,||\cdot||_E)\map(E,w)$ is quasi-compact, it must be must be weakly upper semicontinuous. Consequently, assumption $(\F_3)$ is verified. Moreover, $F$ has nonempty convex and weakly compact values.\par Define $\Lambda\colon\Omega\to{\mathfrak L}(\R{})$ by $\Lambda(t):=(0,t)$ and $\Omega_n:=(0,n)$. Clearly, $\z{\Omega_n}\in{\bm\Omega}(\Lambda)$ and ${\bm\tau}(\Lambda)\neq\varnothing$. It is easily verifiable that functions $g\colon\Omega\times E\to E$ and $k\colon\Delta\to{\mathcal L}(E)$ such that $g(t,u):=S(t)(\mathring{u}_1,\mathring{u}_2)$ and $k(t,s):=S(t-s)$ satisfy assumptions $(\gie_1)$-$(\gie_2)$ and $(\ka_1)$-$(\ka_2)$, respectively. As it comes to verification of assumption \eqref{brzeg}, one may take advantage of the exponential bound of the semigroup $\{S(t)\}_{t\geqslant 0}$ and estimate 
\begin{align*}
\liminf_{n\to\infty}(a_n-\varphi(a_n)-||g(\cdot,0)||_n)&=\liminf_{n\to\infty}\f(a_n-La_n-\sup_{t\in(0,n)}||S(t)(\mathring{u}_1,\mathring{u}_2)||_E\g)\\&\geqslant\liminf_{n\to\infty}\f((1-L)a_n-\sup_{t\in(0,n)}Me^{\omega t}||(\mathring{u}_1,\mathring{u}_2)||_E\g)\\&=\lim_{n\to\infty}\f((1-L)a_n-Me^{\omega n}||(\mathring{u}_1,\mathring{u}_2)||_E\g)=+\infty
\end{align*}
for $a_n:=kMe^{\omega n}||(\mathring{u}_1,\mathring{u}_2)||_E$ with $k>\frac{1}{1-L}$ and $L\in(0,1)$ (the exact values of constants $M,\omega$ have been estimated in \cite{pietkun}).\par In view of Theorem \ref{ex1} the Volterra integral inclusion \eqref{waveeq} possesses a continuous solution $u=(u_1,u_2)\colon(0,\infty)\to E$. A short glimpse at the definition of the semigroup $\{S(t)\}_{t\geqslant 0}$ and the set-valued perturbation $F$ leads to the conclusion that the function $u_2+\mathring{u}_1$ poses a weak solution of the problem \eqref{wave} (compare \cite[Section 7]{thieme}).
\end{proof}
\end{example}

\begin{example}
{\em Consider the following problem:
\begin{equation}\label{multi}
\frac{\partial^Nu}{\partial x_1\ldots\partial x_N}(x)=f(x,u(x))\;\;\;\text{on }\R{N}_+
\end{equation}
whose solutions satisfy the boundary conditions \[u(\sigma(1)x_1,\dots,\sigma(N)x_N)=u_\sigma(x_\sigma)\] for $\sigma\in\{0,1\}^{\{1,\ldots,N\}}\setminus\{\sigma\equiv 1,\sigma\equiv 0\}$ and functions $u_\sigma$ are compatible with each other on the boundary of respective domains $\R{N-\#\sigma^{-1}(0)}_+$. For the sake of typographical simplicity put ${\mathcal I}:=\{0,1\}^{\{1,\ldots,N\}}\setminus\{\sigma\equiv 1,\sigma\equiv 0\}$. Assume that functions $f\colon\R{N}_+\times\R{}\to\R{}$ and $u_\sigma\colon\R{N-\#\sigma^{-1}(0)}_+\to\R{}$ for $\sigma\in{\mathcal I}$ are continuous. Assume also that $|f(x,u)|\<b(x)(1+|u|)$ for some $b\in L^1_{\text{loc}}(\R{N}_+)$. In view of Fubini's theorem the Cauchy problem \eqref{multi} is equivalent to the following Volterra integral equation
\begin{equation}\label{multi2}
u(x)=\sum_{\sigma\in{\mathcal I}}(-1)^{\#\sigma^{-1}(0)+1}u_\sigma(x_\sigma)+(-1)^{N+1}u(\overline{0})+\int\limits_{\prod\limits_{i=1}^N(0,x_i)}f(y,u(y))\,\ell^N(dy),\;\;\;x\in\R{N}_+.
\end{equation}
\par Define $\Lambda\colon\Int(\R{N}_+)\to{\mathfrak L}(\R{N})$ by $\Lambda(x):=\prod\limits_{i=1}^N(0,x_i)$ and $\Omega_n:=(0,n)^N$. Then $\z{\Omega_n}\in{\bm\Omega}(\Lambda)$. Moreover, ${\bm\tau}(\Lambda)\neq\varnothing$ (cf. Example \ref{ex3}(a)). Let
\[\begin{cases}
k(x,y):=1\\ 
F(x,u):=\{f(x,u)\}\\ 
g(x,u):=\sum\limits_{\sigma\in{\mathcal I}}(-1)^{\#\sigma^{-1}(0)+1}u_\sigma(x_\sigma)+(-1)^{N+1}u(\overline{0}),\\
\varphi(x):=Lx\text{ for some }L\in(0,1).
\end{cases}\]
Clearly, assumptions $(\ka_1)$-$(\ka_2)$, $(\gie_1)$-$(\gie_2)$ and $(\F_1)$-$(\F_5)$ are met. Notice that $\#{\mathcal I}=2^N-2$ and \[R_n:=\sum\limits_{\sigma\in{\mathcal I}}\sup_{x\in(0,n)^N}|u_\sigma(x_\sigma)|+|u(\overline{0})|<+\infty.\] Since $||g(\cdot,0)||_n\<R_n$, one obtains for $a_n:=kR_n$ with $k>(1-L)^{-1}$ \[\liminf_{n\to\infty}(a_n-\varphi(a_n)-||g(\cdot,0)||_n)\geqslant\liminf_{n\to\infty}((1-L)k-1)R_n>0\] (just assume that $u(\overline{0})\neq 0$). Summing up, all the hypotheses of Theorem \ref{ex1} are met. Consequently, there exists a continuous function $u\colon\R{N}_+\to\R{}$ for which the integral equation \eqref{multi2} is satisfied. This map poses a classical solution of the initial value problem \eqref{multi}.}
\end{example}

\end{document}